\documentclass{amsart}
\usepackage{amsfonts}
\usepackage{mathrsfs}
\usepackage{amsmath,amsthm,amscd,mathrsfs,amssymb,graphicx,enumerate}
\usepackage{xypic}

\newtheorem{thm}{Theorem}[section]
\newtheorem{cor}[thm]{Corollary}
\newtheorem{lem}[thm]{Lemma}
\newtheorem{prop}[thm]{Proposition}

\theoremstyle{definition}
\newtheorem{defn}[thm]{Definition}
\newtheorem{assm}[thm]{Assumption}
\newtheorem{rmk}[thm]{Remark}

\DeclareMathOperator{\Hom}{Hom} 
 \DeclareMathOperator{\rk}{rk}
 \DeclareMathOperator{\Sym}{Sym}
\DeclareMathOperator{\Pic}{Pic} 

\newcommand{\C}{\ensuremath\mathbb{C}}
\newcommand{\R}{\ensuremath\mathbb{R}}
\newcommand{\Z}{\ensuremath\mathbb{Z}}
\newcommand{\Q}{\ensuremath\mathbb{Q}}

\newcommand{\fa}{\ensuremath\mathfrak{a}}
\newcommand{\fb}{\ensuremath\mathfrak{b}}

\newcommand{\PP}{\ensuremath\mathbb{P}}
\newcommand{\calO}{\ensuremath\mathcal{O}}

\newcommand{\calF}{\ensuremath\mathscr{F}}

\newcommand{\HH}{\ensuremath\mathrm{H}}
\newcommand{\CH}{\ensuremath\mathrm{CH}}

\begin{document}
\title{Surfaces with involution and Prym constructions}
\author{Mingmin Shen}
\thanks{The author is supported by Simons Foundation as a Simons Postdoctoral Fellow.}
\address{DPMMS, University of Cambridge, Wilberforce Road, Cambridge CB3 0WB, UK}
\email{M.Shen@dpmms.cam.ac.uk}


\date{}

\begin{abstract}
An involution on a surface induces involutions on the cohomology,
the Chow group and the Brauer group of the surface. We give a
detailed study of those actions and show that the odd part of these
groups can be used to describe the geometry of cubic fourfolds and
conic bundles over $\PP^3$.
\end{abstract}

\maketitle

\section{Introduction}
When a curve $C$ admits an involution $\sigma:C\rightarrow C$, its
Jacobian $J(C)$ splits into the even part and the odd part. When
$\sigma$ has no fixed point or exactly two fixed points, the odd
part $\mathrm{Pr}(C,\sigma)=\mathrm{Im}\{\sigma-1:J(C)\rightarrow
J(C)\}$ admits a natural principal polarization. The resulting
principally polarized abelian variety is defined to be the
\textit{Prym variety} of $C$ with the involuation $\sigma$. See
\cite{mumford} for more details. Among many interesting
applications, the theory was used to characterize the intermediate
Jacobian of certain types of Fano threefolds and quadrics bundles
(see, for example, \cite{cg, murre1, murre2, beauville}). Together
with the geometry of the theta divisor, such a description was used
to study various rationality questions.

In this paper, we study the case where the curve is replaced by a
surface. Throughout the paper, we will be working over the field
$\C$ of complex numbers. Let $S$ be a smooth projective surface with
$\HH_1(S,\Z)=0$ and $\sigma:S\rightarrow S$ an involution with only
isolated fixed points. Then $\sigma$ induces involutions on the
cohomology $\HH^2(S,\Z)$, the Chow group $\mathrm{A}_0(S)$ of
0-cycles with degree zero and the Brauer group $\mathrm{Br}(S)$.
\begin{defn}[see \cite{pt}]
Let $\Lambda$ be an abelian group together with an involution
$\sigma:\Lambda\rightarrow\Lambda$, then we define the \textit{Prym
part} of $\Lambda$ to be
$\mathrm{Pr}(\Lambda,\sigma)=\mathrm{Im}\{\sigma-1:\Lambda\rightarrow\Lambda\}$.
\end{defn}

With the above definition we can talk about the Prym part of the
cohomology, the Chow group and the Brauer group of $S$. For the
cohomology group, we usually have a geometrically defined
$\sigma$-invariant subgroup $M\subset\HH^2(S,\Z)$ of algebraic
classes. Hence $\sigma$ induces an action on its orthogonal
complement $M^\perp$ and we take the Prym part of $M^\perp$. This
should be viewed as the Prym part of the ``primitive cohomology".
Then this Prym construction can be used to describe the Hodge
structure, the Chow group and the ``second Brauer groups",
$\mathrm{Br}_2(X)$ (see Definition \ref{defn of second Brauer
group}), of cubic fourfolds and conic bundles. To be more precise,
if $X\subset\PP^5$ is a smooth cubic fourfold, then we take $S=S_l$
to be the surface of lines meeting a given general line $l\subset
X$. Then there is a natural involution $\sigma$ on $S$ induced by
taking the residue of two intersecting lines. If $f:X\rightarrow
B=\PP^3$ is a standard conic bundle, then we take $S$ to be the
surface of lines in broken conics. We assume that $S$ is smooth
projective with $\HH_1(S,\Z)=0$. Then there is again an involution
$\sigma$ which is induced by switching the two lines in a broken
conic. In the second case, we also assume that the degree of the
degeneration divisor is odd.  We show that in both cases the Hodge
structure and the Chow group of $X$ can be recovered as the Prym
part of the corresponding structures of the surface $S$. The group
$\mathrm{Br}_2(X)$ can also be recovered in generic cases. When $X$
is special, there might be a 2-torsion subgroup $K$ of
$\mathrm{Br}_2(X)$ which does not appear in the Prym part of the
Brauer group of $S$. In the case when $X$ is a cubic fourfold, this
group $K$ is independent of the choice of the general line $l\subset
X$. So far we do not have an example where the group $K$ is
nontrivial.

The idea of using the space of lines on a hypersurface to describe
the Hodge structure has been studied by many mathematicians. In
\cite{shimada}, I.~Shimada considered the total space of lines on a
general hypersurface and show that the associated cylinder
homomorphism is an isomorphism in many cases. In \cite{izadi},
E.~Izadi gave a Prym construction for the Hodge structure of cubic
hypersurfaces. Our result on cubic fourfolds gives such a
construction for the Chow group and the Brauer group as well.
Another advantage of our approach is that are potentially
applications of Theorem \ref{thm on prym construction} to other
situations. See \cite{pt} for more general Prym-Tjurin construction.

The plan of this article is as follows. In Section 2, we give a
detailed study of the action of $\sigma$ on $\HH^2(S,\Z)$. The main
results are summarized in Theorem \ref{thm on surface with
involution} and Theorem \ref{thm on surface with involution fixed
point free}. It turns out that the structure of this action is quite
sensitive to whether $\sigma$ has fixed points or not. In section 3,
we study the Prym part of the cohomology, Chow group and Brauer
group of $S$. Sections 4 and 5 are devoted the application to the
case of cubic fourfolds. Section 6 studies the case of conic bundles
over $\PP^3$.

\textbf{Notation and conventions.}

(1) Let $\Lambda$ be an integral Hodge structure of weight $2k$.
Then we use $\mathrm{Hdg}(\Lambda)=\Lambda^{k,k}\cap\Lambda$ to
denote the group of integral Hodge classes. If $\Lambda$ comes from
geometry, we also use
$\mathrm{Alg}(\Lambda)\subset\mathrm{Hdg}(\Lambda)$ to denote the
subgroup of algebraic classes. When $\Lambda$ is polarized (which
boils down to an integral symmetric bilinear form), we define
$\Lambda_\mathrm{tr}\subset\Lambda$ to be the orthogonal complement
of $\mathrm{Hdg}^{2k}(\Lambda)$. We also use $T(\Lambda)$ to denote
the quotient of $\Lambda$ by $\mathrm{Hdg}(\Lambda)$. If
$\Lambda=\HH^{2k}(Y,\Z)/\mathrm{tor}$ comes from geometry, we also
use $\mathrm{Hdg}^{2k}(Y)$, $\mathrm{Alg}^{2k}(Y)$ and $T^{2k}(Y)$
to denote $\mathrm{Hdg}(\Lambda)$, $\mathrm{Alg}(\Lambda)$ and
$T(\Lambda)$ respectively.

(2) In all sections except Section 5, we use $G=\{1,\sigma\}$ to
denote the cyclic group of order 2. For any abelian group $A$, we
use $A_{+}$ to denote the $G$-module $A$ on which $\sigma$ acts as
$+1$ and $A_{-}$ to denote the $G$-module $A$ on which $\sigma$ acts
as $-1$.

(3) Let $V$ be a vector space of dimension $n$, and $1\leq
r_1<\ldots< r_m<n$ a sequence of increasing integers. We use
$G(r_1,\ldots,r_m,V)$ to be the flag variety parameterizing flags
$V_1\subset\cdots\subset V_{m}\subset V$ with $\dim V_i=r_i$. If
$m=1$ and $r_1=1$, then we use $\PP(V)$ to denote $G(1,V)$. If we
replace $V$ by a vector bundle $\mathscr{E}$, then we define the
relative flag varieties in a similar way.

(4) For any morphism $f:Y\rightarrow X$ between smooth varieties
which is birational onto image, we define the normal bundle
$\mathscr{N}_{Y/X}$ of $Y$ in $X$ to be the cokernel of the
homomorphism $df:T_Y\rightarrow f^*T_X$.

\section{Surfaces with involution}
In this section, we fix $S$ to be a smooth projective surface over
$\C$. Assume that there is an involution $\sigma:S\to S$ with
finitely many isolated fixed points $x_1,\ldots,x_r\in S$. This
section is devoted to the study of the action of $\sigma$ on the
cohomology group $H^2(S,\Z)$.

Let $Y=S/\sigma$ be the quotient. Then $Y$ is a projective surface
with finitely many ordinary double points $y_1,\ldots,y_r\in Y$. Let
$\pi:S\to Y$ be the natural morphism of degree 2. We use $\tilde{S}$
to denote the blow up of $S$ at the points $x_1,\ldots,x_r$ with
$E'_i$ being the corresponding exceptional curves. We use
$\tilde{Y}$ to denote the minimal resolution of $Y$ with $E_i$ being
the exceptional divisors. Note that the $E_i$'s are $(-2)$-curves on
$\tilde{Y}$. Then $\sigma$ induces an involution, still denoted by
$\sigma$, on $\tilde{S}$ with all points in the $E'_i$'s being
fixed. One also easily sees that $\tilde{Y}$ is naturally the
quotient of $\tilde{S}$ by the action of $\sigma$. If we use
$\tilde{\pi}:\tilde{S}\to\tilde{Y}$ to denote the quotient map, then
we have the following commutative diagram,
$$
\xymatrix{\tilde{S}\ar[r]^g\ar[d]_{\tilde{\pi}} &S\ar[d]^{\pi}\\
   \tilde{Y}\ar[r]^f &Y}
$$
where $f$ and $g$ are the blow-up's.

On a complex analytic space $X$, we use $\mathcal{H}_p(X)$ to denote
the sheaf associated to
$$
U\longmapsto \HH_p(X,X-U).
$$
Note that the stalk of $\mathcal{H}_p(X)$ at smooth point $x\in X$
is, by excision, equal to $\HH_p(B,\partial B)$ where $B\subset X$
is a small closed ball centered at $x$. Hence $\mathcal{H}_p(X)$,
for $p\neq 2\dim_{\C} X$, is supported on the singular locus of $X$.

\begin{lem}\label{local homology lemma}
Let $Y$ be the quotient of $S$ by $\sigma$ as above, then

(i) $\mathcal{H}_p(Y)=0$, for $p\neq 2,4$.

(ii) $\mathcal{H}_2(Y)$ is the skyscraper sheaf supported on $y_i$'s
with $\mathcal{H}_2(Y)_{y_i}\cong\Z/2\Z$.

(iii) $\mathcal{H}_4(Y)$ is the constant sheave isomorphic to
$\HH_4(Y,\Z)\cong\Z$.
\end{lem}
\begin{proof}
We know that $\mathcal{H}_p(Y)$ is zero at smooth points of $Y$ for
$p\neq 4$. By assumption, we can embed $Y$ in some projective space.
Let $B_i\subset Y$ be the intersection of $Y$ and a small closed
ball (in the ambient projective space) centered at $y_i$. By
construction, $B_i$ is contractible. By excision, we have
isomorphism
$$
\mathcal{H}_p(Y)_{y_i}\cong\HH_p(B_i,\partial B_i)
$$
Consider the long exact sequence associated to the pair
$(B_i,\partial B_i)$,
$$
\xymatrix{\cdots\ar[r] &\HH_p(\partial B_i)\ar[r] &\HH_p(B_i)\ar[r]
&\HH_p(B_i,\partial B_i)\ar[r] &\HH_{p-1}(\partial B_i)\ar[r]
&\cdots }
$$
where all cohomology groups have integral coefficients. Since $B_i$
is contractible, we have $\HH_p(B_i,\Z)=0$ for $p\neq 0$ and
$\HH_0(B_i,\Z)=\Z$. It follows that
$$
\HH_p(B_i,\partial B_i) =\HH_{p-1}(\partial B_i,\Z), \quad p=2,3,4
$$
and we also have $\HH_p(B_i,\partial B_i)=0$ for $p=0,1$. Let
$\tilde{B}_i\subset \tilde{Y}$ be the inverse image of $B_i$. Then
$\tilde{B}_i$ can be thought of as a tubular neighbourhood of
$E_i\cong\PP^1$ in $\tilde{Y}$. Then $\partial B_i=\partial
\tilde{B}_i$ is homeomorphic to the unit circle bundle over $E_i$
sitting inside $\mathscr{N}_{E_i/\tilde{Y}}\cong\calO_{\PP^1}(-2)$
(equipped with some hermitian metric). Hence $\partial B_i$ can be
viewed as a circle bundle over $\PP^1$ whose Euler class
$e_i=-2\in\HH^2(E_i,\Z)$. Consider the associated Gysin sequence
$$
\xymatrix{0=\HH^1(E_i)\ar[r] &\HH^1(\partial B_i)\ar[r]
&\HH^0(E_i)\ar[r]^{\cup e_i} &\HH^2(E_i)\ar[r] &\HH^2(\partial
B_i)\ar[r] &0}
$$
where all cohomology groups have integral coefficients. One easily
finds $\HH^1(\partial B_i,\Z)=0$ and $\HH^2(\partial
B_i,\Z)=\Z/2\Z$. Note that $\partial B_i$ is an orientable
3-dimensional compact manifold. By Poincar\'e duality, we get
$$
\HH_p(\partial B_i,\Z)=\begin{cases}
 0, &p=2\\
 \Z, &p=0,3\\
 \Z/2\Z, &p=1
\end{cases}
$$
Hence we have
$$
\mathcal{H}_p(Y)_{y_i}=\begin{cases}
 0, &p=0,1,3\\
 \Z/2\Z, &p=2\\
 \Z, &p=4
\end{cases}
$$
The lemma follows easily from this computation.
\end{proof}

On a complex analytic space $X$, we have Zeeman's spectral sequence
$$
E^2_{p,q}=\HH^q(X,\mathcal{H}_p(X))\Longrightarrow \HH_{p-q}(X,\Z)
$$
which measures the failure of Poincar\'e duality, see
\cite{zeeman1}\footnote{Make references more precise.}. The
differential $d^r$ goes as follows
$$
d^r:E^r_{p,q}\longrightarrow E^r_{p+r-1,q+r}
$$
The convergence means that there is a filtration
$$
H_s(X)=F_s^0\supseteq F_s^1\supseteq \cdots \supseteq
F_s^{n-s}\supseteq 0,\quad n=\dim_\R(X)
$$
such that $E^{\infty}_{s+q,q}\cong F^q_s/F^{q+1}_s$.

We apply Zeeman's spectral sequence to the singular surface $Y$, and
note that $E^2_{p,q}=0$ unless $p=4$ or $(p,q)=(2,0)$. This means
that $E^3=E^2$ and we have the only nontrivial differential map
$$
d^3:\HH^0(Y,\mathcal{H}_2(Y))\cong(\Z/2\Z)^r\rightarrow \HH^3(Y,\Z)
$$
We set
\begin{equation}\label{defn of M and N}
M=\ker(d^3),\qquad N=\mathrm{im}(d^3).
\end{equation}
Then from the spectral sequence, we get the following

\begin{prop}\label{failure of PD}
Let $Y$ be a projective surface with only isolated ordinary double
points $y_1,\ldots,y_r\in Y$ as above. Then the following statements
are true.

(i) $\cap [Y]:\HH^i(Y,\Z)\to\HH_{4-i}(Y,\Z)$ is an isomorphism for
$i=0,1,4$.

(ii) There is an exact sequence
$$
\xymatrix{
 0\ar[r] &\HH^2(Y)\ar[r]^{\cap[Y]} &\HH_2(Y)\ar[r] &\HH^0(Y,\mathcal{H}_2(Y))
 \ar[r]^{\quad d^3} &\HH^3(Y)\ar[r]^{\cap[Y]} &\HH_1(Y)\ar[r] &0
}
$$
where all (co)homology groups have integral coefficients.
\end{prop}

\begin{rmk}\label{remark on Leray}
If we split the above exact sequence in (ii), we get the following
short exact sequences,
\begin{equation}
\xymatrix{
 0\ar[r] & \HH^2(Y,\Z)\ar[r]^{\cap[Y]} &\HH_2(Y,\Z)\ar[r] &M\ar[r] &0
}
\end{equation}
and
\begin{equation}
\xymatrix{
 0\ar[r] &N\ar[r] & \HH^3(Y,\Z)\ar[r]^{\cap[Y]} &\HH_1(Y,\Z)\ar[r] &0
}
\end{equation}
Consider the Leray spectral sequence associated to the resolution
$f:\tilde{Y}\to Y$, i.e.
$$
E_2^{p,q}=\HH^p(Y,R^qf_*\Z)\Longrightarrow \HH^{p+q}(\tilde{Y},\Z).
$$
Note that
$$
R^qf_*\Z=\begin{cases}
 \Z, &q=0;\\
 0, &q=1;\\
 \oplus_{i=1}^r \HH^2(E_i,\Z)_{y_i}, &q=2.
\end{cases}
$$
From this we get that
\begin{equation}
f^*: \HH^i(Y,\Z)\to\HH^i(\tilde{Y},\Z)
\end{equation}
is an isomorphism for $i=0,1,4$. There is also an exact sequence
$$
\xymatrix{
 0\ar[r] &\HH^2(Y)\ar[r]^{f^*} &\HH^2(\tilde{Y})\ar[r]
 & \oplus_{i=1}^{r}\HH^2(E_i)\ar[r]^{\quad\alpha}
 &\HH^3(Y)\ar[r]^{f^*} &\HH^3(\tilde{Y})\ar[r] &0
}
$$
where all cohomology groups have integral coefficients. This
sequence is compatible with the one in (ii) of Proposition
\ref{failure of PD} in the following sense. Using Poincar\'e duality
on $\tilde{Y}$, we can define
$$
f_*:\HH^i(\tilde{Y},\Z)\cong \HH_{4-i}(\tilde{Y},\Z)\rightarrow
\HH_{4-i}(Y,\Z).
$$
Then $f_*f^*=\cap[Y]:\HH^i(Y,\Z)\rightarrow \HH_{4-i}(Y,\Z)$. If we
identify $\HH^0(Y,\mathcal{H}_2(Y))$ with
$\oplus_{r=1}^{r}\HH^2(\partial B_i,\Z)$ as in the proof of Lemma
\ref{local homology lemma}, then we have the following commutative
diagram
$$
\xymatrix{
 0\ar[r] &\HH^2(Y)\ar[r]^{f^*}\ar@{=}[d] &\HH^2(\tilde{Y})\ar[r]\ar[d]^{f_*}
 & \oplus_{i=1}^{r}\HH^2(E_i)\ar[r]^{\quad\alpha}\ar[d]^{\beta}
 &\HH^3(Y)\ar[r]^{f^*}\ar@{=}[d] &\HH^3(\tilde{Y})\ar[r]\ar[d]^{f_*} &0\\
 0\ar[r] &\HH^2(Y)\ar[r]^{\cap[Y]} &\HH_2(Y)\ar[r] &\oplus_{i=1}^{r}
 \HH^2(\partial B_i) \ar[r]^{\quad d^3} &\HH^3(Y)\ar[r]^{\cap[Y]}
 &\HH_1(Y)\ar[r] &0
}
$$
where all (co)homology groups have integral coefficients. Here the
homomorphism $\beta$ is nothing but the pull back induced from the
circle bundle structure $\partial B_i\to E_i$. The Gysin sequence
tells us that $\ker(\beta)$ is generated by the Euler classes
$e_i=-2\in\HH^2(E_i,\Z)$. By diagram chasing, we easily get that
$$f_*:\HH^2(\tilde{Y},\Z)\to\HH_2(Y,\Z)$$ is surjective and
$$f_*:\HH^3(\tilde{Y},\Z)\to\HH_1(Y,\Z)$$ is an isomorphism. Then it
follows that
\begin{equation}
N\cong \mathrm{im}(\alpha).
\end{equation}
\end{rmk}

Before moving on, we recall some basic facts about group cohomology.
Let $G$ be a finite group. The group cohomology functors
$\HH^i(G,-)$ are defined to be the right derived functors of the
functor
$$
M\mapsto M^G=\{x\in M: gx=x,\forall g\in G\}
$$
on the category of $G$-modules. If $G=\{1,\sigma\}$ is the cyclic
group of order $2$, then the group cohomology can be written
explicitly as
$$
\HH^i(G,M) =\begin{cases}
 M^G, & i=0;\\
 \frac{ \{x\in M: \sigma x=-x\} }{(\sigma-1)M}, & i\text{ is odd};\\
 \frac{ \{x\in M: \sigma x=x\} }{(\sigma+1)M}, &i \text{ is even}.
\end{cases}
$$
where $M$ is a $G$-module. This implies that $\HH^i(G,M)$ is of
2-torsion for all positive $i$. For example, if
$\alpha\in\HH^1(G,M)$, then $\alpha$ can be represented by
$\tilde{\alpha}\in M$ such that
$\sigma(\tilde{\alpha})=-\tilde{\alpha}$. Then
$2\tilde{\alpha}=\tilde{\alpha}-\sigma(\tilde{\alpha})\in
(\sigma-1)M$. This means that $2\alpha = 0$. We state this as a
lemma for future reference.
\begin{lem}\label{lem 2-torsion}
Let $G$ be the cyclic group of order 2 and $M$ a $G$-module, then
$\HH^i(G,M)$ is of 2-torsion for all $i>0$. If $M$ is a torsion group with no 2-torsion elements, then $\HH^i(G,M)=0$ for all $i>0$.
\end{lem}

The surface $S$ can be viewed as a space with $G$-action, where
$G=\{1,\sigma\}$ is the cyclic group of order 2. On the category of
$G$-sheaves on $S$, we have the $G$-invariant global section functor
$\Gamma^G(\calF)=\Gamma(S,\calF)^G$. The right derived functor
$R^p(\Gamma^G)(\calF)$ is also denoted by $\HH^p(G;S,\calF)$. Then
we have two natural spectral sequences attached to this situation,
\begin{equation}\label{Grothendieck spectral sequence I}
_IE_2^{p,q}=\HH^p(Y,R^q(\pi_*^G)\calF)\Longrightarrow\HH^{p+q}(G;S,\calF)
\end{equation}
and
\begin{equation}\label{Grothendieck spectral sequence II}
_{II}E_2^{p,q}=\HH^p(G,\HH^q(S,\calF))\Longrightarrow\HH^{p+q}(G;S,\calF)
\end{equation}
See chapter V of \cite{grothendieck} for more details. We apply the
first spectral sequence to the case $\calF=\Z$, and note that
$$
R^q(\pi^G_*)\Z=\begin{cases}
 \Z, &q=0;\\
 \oplus_{i=1}^{r}\HH^q(G,\Z_{x_i})_{y_i}, &q>0\text{ even};\\
 0, &\text{otherwise}.
\end{cases}
$$
It follows that all the differential homomorphisms
$$
d_{2}^{p,q}: _IE^{p,q}\longrightarrow _IE^{p+2, q-1}
$$
in the spectral sequence $_IE_2$ are zero except for
$$
d_2=d_2^{2,0}: \oplus_{i=1}^{r}\HH^2(G,\Z_{x_i}) \rightarrow
\HH^3(Y,\Z).
$$

\begin{lem}\label{first spectral lemma}
(i) There is a natural isomorphism $\HH^2(G,\Z_{x_i})\cong
\mathcal{H}_2(Y)_{y_i}$ such that the differential map $d_2$ is
identified with $d^3$ in Zeeman's spectral sequence.

(ii) $\HH^i(Y,\Z)\cong \HH^i(G;S,\Z)$ for $i=0,1$.

(iii) We have the following short exact sequences
\begin{align*}
\xymatrix{0\ar[r] &\HH^4(Y,\Z)\ar[r] &\HH^4(G;S,\Z)\ar[r] &\oplus \HH^4(G,\Z_{x_i})\ar[r] &0 }\\
\xymatrix{0\ar[r] &\HH^2(Y,\Z)\ar[r] &\HH^2(G;S,\Z)\ar[r] &M\ar[r] &0}\\
\xymatrix{0\ar[r] &N\ar[r] &\HH^3(Y,\Z)\ar[r] &\HH^3(G;S,\Z)\ar[r] &0}
\end{align*}

(iv) $\HH^i(G;S,\Z)\cong \HH_{4-i}(Y,\Z)$ for $i=0,1,2,3$.
\end{lem}

\begin{proof}
(i) Let $B_i$ be the intersection of $Y$ and a small closed ball
centered at $y_i$ as before. Let $B'_i\subset S$ be the inverse
image of $B_i$. Then $B'_i$ is a small neighbourhood of $x_i\in S$
and $B'_i$ is homeomorphic to a 4-dimensional ball. The restriction
of $\pi$ to $\partial B'_i$ gives a covering map, $\pi_i:\partial
B'_i=S^3\to\partial B_i$, of degree 2. The spectral sequence $_IE$
for the covering map $\pi_i$ gives a natural isomorphism
$$
\HH^p(G;\partial B'_i,\Z)\cong \HH^p(\partial B_i,\Z),\quad\forall
p\geq 0.
$$
Apply the second spectral sequence $_{II}E$ to the covering map
$\pi_i$, we get
$$
\HH^2(G;\partial B'_i,\Z)\cong \HH^2(G,\HH^0(\partial B'_i,\Z))\cong \HH^2(G,\Z_{x_i})
$$
Hence we have a natural isomorphism $\HH^2(\partial B_i,\Z)\cong
\HH^2(G,\Z_{x_i})$. (ii) follows directly from the spectral sequence
$_IE$. To prove (iii) and (iv), we only need to compare the spectral
sequence $_IE$ with Zeeman's spectral sequence.
\end{proof}

From now on, we assume the following
\begin{assm}\label{vanishing cohomology assumption}
$\HH_1(S,\Z)=0$.
\end{assm}
One easily sees that the assumption above implies $\HH^1(S,\Z)=0$ by
the universal coefficient theorem and $\HH^3(S,\Z)=0$ by the
Poincar\'e duality. Consider the second spectral sequence $_{II}E$.
Since $\HH^q(S,\Z)=\Z$, $q=0,4$, with the trivial action of $G$, we
have
$$
\HH^p(G,\HH^q(S,\Z))=\begin{cases}
 \Z, & p=0;\\
 \Z/2\Z, &p=2k,k\geq 1;\\
 0, & p=2k-1, k\geq 1,
\end{cases}
$$
where $q=0,4$.

\begin{lem} \label{second spectral lemma}
Under the assumptions \ref{vanishing cohomology assumption}, the
following statements are true.

(i) $\HH^0(G;S,\Z)=\Z$ and $\HH^1(G;S,\Z)=0$.

(ii) There is a short exact sequence
$$
\xymatrix{0\ar[r] &\HH^2(G,\HH^0(S,\Z))=\Z/2\Z\ar[r]
&\HH^2(G;S,\Z)\ar[r] &\HH^2(S,\Z)^G\ar[r] &0}
$$

(iii) If $r\geq 1$, then $\HH^3(G;S,\Z)\cong\HH^1(G,\HH^2(S,\Z))$.

(iv) If $r\geq 1$, then there is a short exact sequence
$$
\xymatrix{ 0\ar[r] &\HH^4(G,\HH^0(S,\Z)) \ar[r] &
\HH^4(G;S,\Z)_{\mathrm{tor}} \ar[r] &\HH^2(G,\HH^2(S,\Z))\ar[r] &0 }
$$
\end{lem}

\begin{proof}
This is essentially a restatement of the spectral sequence $_{II}E$.
(i) and (ii) are easy. For (iii) and (iv), we note that the spectral
sequence gives an exact sequence
$$
\xymatrix{
 0\ar[r] &\HH^3(G;S,\Z)\ar[r] &\HH^1(G,\HH^2(S,\Z))\ar[r]^\rho & \HH^4(G,\HH^0(S,\Z))\ar[r]^\tau &\HH^4(G;S,\Z)
}
$$
The composition of the natural map
$\HH^4(G;S,\Z)\to\oplus_{i=1}^{r}\HH^4(G,\Z_{x_i})$ (see Lemma
\ref{first spectral lemma}) with $\tau$ is the natural map induced
by the restriction homomorphism $\HH^4(S,\Z)\rightarrow \oplus
\Z_{x_i}$. One sees easily that this composition is nonzero. It
follows that $\tau$ is nonzero and hence injective. Thus $\rho=0$.
Then (iii) and (iv) are easily deduced from the spectral sequence.
\end{proof}

\begin{thm}\label{thm on surface with involution}
Let $S$ be a smooth complex algebraic surface with an involution
$\sigma:S\to S$. Assume that $\sigma$ has finitely many fixed points
$x_1,\ldots,x_r\in S$, $r\geq 1$. Let $Y=S/\sigma$ be the quotient
of $S$ by the involution. If $\HH_1(S,\Z)=0$, then the following are
true.

(i) The homology groups of $Y$ are given by
\begin{align*}
&\HH_0(Y,\Z)=\Z,\qquad \HH_1(Y,\Z)=0, \qquad \HH_2(Y,\Z)=\Z/2\Z\oplus \HH^2(S,\Z)^G\\
 &\HH_3(Y,\Z)=0,  \qquad \HH_4(Y,\Z)=\Z.
\end{align*}

(ii) The cohomology groups of $Y$ are given by
\begin{align*}
&\HH^0(Y,\Z)=\Z,\qquad \HH^1(Y,\Z)=0, \qquad \HH^2(Y,\Z)\subset \HH^2(S,\Z)^G\\
 &\HH^3(Y,\Z)=\Z/2\Z,  \qquad \HH^4(Y,\Z)=\Z.
\end{align*}

(iii) The failure of Poincar\'e duality on $Y$ are measured by
\begin{align*}
M=\mathrm{coker}\{\cap[Y]:\HH^2(Y,\Z)\to\HH_2(Y,\Z)\} &\cong(\Z/2\Z)^{r-1},\\
N=\ker\{\cap[Y]:\HH^3(Y,\Z)\to\HH_1(Y,\Z)\} &\cong \Z/2\Z.
\end{align*}

(iv) We always have $r\geq 2$ and
$\mathrm{coker}\{\HH^2(Y,\Z)\rightarrow\HH^2(S,\Z)^G\}\cong
(\Z/2\Z)^{r-2}$.

(v)\,\, $\HH^1(G,\HH^2(S,\Z))=0$ and
$\HH^2(G,\HH^2(S,\Z))\cong(\Z/2\Z)^{r-2}$.

(vi) We have the following isomorphism as $G$-modules
$$
\HH^2(S,\Z)\cong \Z[G]^{r_0}\oplus \Z_{+}^{r-2}.
$$
where $r_0=\frac{h^2(S) -r+2}{2}$.
\end{thm}

\begin{proof}
First we note that the assumption $\HH_1(S,\Z)=0$ implies that
$\HH^1(S,\Z)=0$, $\HH^2(S,\Z)$ is torsion free (the universal
coefficients) and $\HH^3(S,\Z)=0$ (the Poincar\'e duality). Since
the composition of the natural maps
$$
\varphi:\HH^2(G,\HH^0(S,\Z))\rightarrow \HH^2(G;S,\Z)\rightarrow
\oplus_{i=1}^{r}\HH^2(G,\Z_{x_i})
$$
comes from the natural diagonal/restriction map
$\HH^0(S,\Z)=\Z\rightarrow \oplus \Z_{x_i}=\Z^r$, it follows that
$\varphi$ is injective if $r\geq 1$. Actually, under the
isomorphisms $\HH^2(G,\HH^0(S,\Z))\cong\Z/2\Z$ and
$\oplus\HH^2(G,\Z_{x_i})\cong(\Z/2\Z)^r$, the homomorphism $\varphi$
is nothing but the diagonal map. Consider the diagram
\begin{equation}
\xymatrix{
 & &(\Z/2\Z)^r & &\\
0\ar[r] &\Z/2\Z\ar[r]\ar[ur]^{\varphi} &\HH_2(Y,\Z) \ar[u]\ar[r] &\HH^2(S,\Z)^G\ar[r] &0\\
 & &\HH^2(Y,\Z)\ar@{^(->}[u]\ar[ur]_{\pi^*} & &
}
\end{equation}
One easily sees that the injectivity of $\varphi$ implies that of
$\pi^*$. Hence $\HH^2(Y,\Z)$ is torsion free.

Now we start to prove (i). By (iv) of Lemma \ref{first spectral
lemma} and (i) of Lemma \ref{second spectral lemma}, we get
$$
\HH^1(Y,\Z)=\HH_3(Y,\Z)=\HH^1(G;S,\Z)=0
$$
Similarly, by (iv) of Lemma \ref{first spectral lemma} and (ii) of
Lemma \ref{second spectral lemma}, we get
$$
\HH_2(Y,\Z)=\HH^2(G;S,\Z)=\HH^2(S,\Z)^G\oplus\Z/2\Z
$$
We also note that (ii) of Lemma \ref{first spectral lemma}, together
with (i) of Lemma \ref{second spectral lemma}, implies that
$\HH^1(Y,\Z)=0$. We have already seen that $\HH^2(Y,\Z)$ is torsion
free. Hence the universal coefficient theorem tells us that
$\HH_1(Y,\Z)=0$.

The conclusions in (ii) follow from (i) by the universal coefficient
theorem.

To prove (iii), we first note that $\HH_1(Y,\Z)=0$, together with
the exact sequence in (ii) of Proposition \ref{failure of PD},
implies that $N=\HH^3(Y,\Z)=\Z/2\Z$. Since $M$ and $N$ fit into the
following short exact sequence by \eqref{defn of M and N},
$$
\xymatrix {0\ar[r] &M\ar[r]
&\oplus\mathcal{H}(Y)_{y_i}=(\Z/2\Z)^r\ar[r] &N\ar[r] &0, }
$$
we get $M\cong(\Z/2\Z)^{r-1}$.

To prove (iv), we consider the following diagram
$$\xymatrix{
0\ar[r] &\Z/2\Z\ar[r]^{\varphi'}\ar@{=}[d] &M\ar[r] &(\Z/2\Z)^{r-2}\ar[r] &0\\
0\ar[r] &\Z/2\Z\ar[r] &\HH_2(Y,\Z) \ar[u]\ar[r] &\HH^2(S,\Z)^G\ar[r]\ar[u] &0\\
 & &\HH^2(Y,\Z)\ar[u]\ar@{=}[r] &\HH^2(Y,\Z)\ar[u]_{\pi^*} &
}
$$
Note that $\HH^2(Y,\Z)$ is torsion free. This forces $\varphi'$ to
be injective. Hence $M\neq 0$ and $r\geq 2$. The last column shows
that $\mathrm{coker}(\pi^*)\cong(\Z/2\Z)^{r-2}$.

By (iii) of Lemma \ref{second spectral lemma}, we get
$$
\HH^1(G,\HH^2(S,\Z))=\HH^3(G;S,\Z)=\HH_1(Y,\Z)=0.
$$
This proves the first half of (v). Consider the following diagram
$$\xymatrix{
0\ar[r] &\HH^4(Y,\Z)\ar[r]\ar@{=}[d] &\HH^4(S,\Z)^G\ar[r] &\Z/2\Z\ar[r] &0\\
0\ar[r] &\HH^4(Y,\Z)\ar[r] &\HH^4(G;S,\Z) \ar[u]\ar[r] &\oplus\HH^4(G,\Z_{x_i})\ar[r]\ar[u] &0\\
 & &\HH^4(G;S,\Z)_{\mathrm{tor}}\ar[u]\ar@{=}[r] &\HH^4(G;S,\Z)_{\mathrm{tor}}\ar[u] &
}
$$
We easily get $\HH^4(G;S,\Z)_{\mathrm{tor}}\cong(\Z/2\Z)^{r-1}$. By
(iv) of Lemma \ref{second spectral lemma}, we get
$$
\HH^2(G,\HH^2(S,\Z))=\mathrm{coker}\{\Z/2\Z\rightarrow
\HH^4(G;S,\Z)_{\mathrm{tor}}\}=(\Z/2\Z)^{r-2}
$$
This proves the second half of (v).

The statement (vi) follows from Lemma \ref{structure lemma on
G-modules} on structure of $G$-modules.
\end{proof}

We will be mainly interested in the case when $\sigma$ has fixed
points, but to complete the picture, we state the corresponding
results on the case where $\sigma$ is fixed point free.
\begin{thm}\label{thm on surface with involution fixed point free}
Let $S$ be a smooth projective surface with $\HH_1(S,\Z)=0$. Let
$\sigma:S\rightarrow S$ be an involution that has no fixed point and
$G=\{1,\sigma\}$. Let $\pi:S\rightarrow Y$ be the quotient morphism,
which is \'etale
of degree 2. Then the following are true.\\
(i) The homology groups of $Y$ are given by
\begin{align*}
&\HH_0(Y,\Z)=\Z;\qquad \HH_1(Y,\Z)=\Z/2\Z;\qquad
\HH_2(Y,\Z)=\HH_2(S,\Z)^G\oplus\Z/2\Z;\\
&\HH_3(Y,\Z)=0;\qquad \HH_4(Y,\Z)=\Z.
\end{align*}
(ii) The cohomology groups of $Y$ are given by
\begin{align*}
&\HH^0(Y,\Z)=\Z;\qquad \HH^1(Y,\Z)=0;\qquad
\HH^2(Y,\Z)=\HH^2(S,\Z)^G\oplus\Z/2\Z;\\
&\HH^3(Y,\Z)=\Z/2\Z;\qquad \HH^4(Y,\Z)=\Z.
\end{align*}
(iii) As a $G$-module, the group cohomology of $\HH^2(S,\Z)$ is
given by the following
$$
\HH^1(G,\HH^2(S,\Z))=(\Z/2\Z)^2;\qquad\HH^2(G,\HH^2(S,\Z))=0.
$$
(iv) We have the following isomorphism of $G$-modules
$$
\HH^2(S,\Z)=\Z[G]^{r_0}\oplus\Z_{-}^2.
$$
where $r_0=\frac{h^2(S)-2}{2}$. In particular, the second Betti
number of $S$ is even.
\end{thm}
\begin{proof}
The natural quotient morphism $\pi:S\rightarrow Y$ is \'etale. This
implies that
$$
R^q\pi^G_*\Z=\begin{cases}
 \Z, &q=0;\\
 0, &q>0.
\end{cases}
$$
Hence the first spectral sequence $_IE_2$ of \eqref{Grothendieck
spectral sequence I} degenerates and we have
$$
\HH^i(G;S,\Z)=\HH^i(Y,\Z),\;\forall i\geq 0.
$$
By (i) of Lemma \ref{second spectral lemma}, we get $\HH^0(Y,\Z)=\Z$
and $\HH^1(Y,\Z)=0$. By (ii) of Lemma \ref{second spectral lemma},
we get $\HH^2(Y,\Z)=\HH^2(S,\Z)^G\oplus\Z/2\Z$. Then by the
universal coefficient theorem for cohomology, we know that
$\HH_1(Y,\Z)=\Z/2\Z$. This implies that $\HH^3(Y,\Z)=\Z/2\Z$ by
Poincar\'e duality on $Y$. This finishes the proof of statement
(ii). Then (i) follows automatically by Poincar\'e duality. The
spectral sequence $_{II}E$ in \eqref{Grothendieck spectral sequence
II} gives an exact sequence
$$
\xymatrix{
 0\ar[r] &\HH^3(Y,\Z)\ar[r] &\HH^1(G,\HH^2(S,\Z))\ar[r]
 &\HH^4(G,\HH^0(S,\Z))\ar[r]^\alpha &\HH^4(Y,\Z)=\Z
}
$$
Since $\HH^4(G,\HH^0(S,\Z))=\Z/2\Z$, the last homomorphism $\alpha$
is zero. Then we can conclude that $\HH^1(G,\HH^2(S,\Z))=(\Z/2\Z)^2$
since this group is always $2$-torsion by Lemma \ref{lem 2-torsion}.
The above long exact sequence goes further as
$$
\xymatrix{
 0\ar[r] &\HH^2(G,\HH^2(S,\Z))\ar[r] &\HH^4(Y,\Z)\ar[r]
 &\HH^4(S,\Z)\ar[r] &\Z/2\Z\ar[r] &0.
}
$$
Since $\HH^4(Y,\Z)$ is torsion free, $\HH^2(G,\HH^2(S,\Z))$ has to
be zero since it is always 2-torsion by Lemma \ref{lem 2-torsion}.
This proves (iii). The statement (iv) follows from the following
lemma.
\end{proof}

\begin{lem}\label{structure lemma on G-modules}
Let $G$ be the cyclic group of order 2 with generator $\sigma$.
Assume that $G$ acts on a finitely generated free abelian group
$\Lambda$. Then we have
$$
\Lambda=\Z[G]^{r_0}\oplus \Z_{+}^{r_+}\oplus \Z_{-}^{r_-}
$$
as $G$-modules, where $\Z_{-}$ is the abelian group $\Z$ on which
$\sigma$ acts as multiplication by $-1$ and $\Z_{+}$ is the trivial
$G$-module. In particular, $\mathrm{rk}(\Lambda)=2r_0+r_{+}+r_{-}$,
$\HH^1(G,\Lambda)\cong(\Z/2\Z)^{r_{-}}$ and
$\HH^2(G,\Lambda)=(\Z/2\Z)^{r_+}$.
\end{lem}
\begin{proof}
Let $B=\ker(\sigma-1)=\Lambda^G\subset\Lambda$ and
$P=\ker(\sigma+1)\subset\Lambda$. Consider the exact sequence,
$$
\xymatrix{ 0\ar[r] &B\oplus P\ar[r] &\Lambda\ar[r] &Q\ar[r] &0}
$$
For each element $q\in Q$, let $\tilde{q}\in \Lambda$ be a lift of
$q$. Since
$$
2\tilde{q}=(1+\sigma)\tilde{q}+(1-\sigma)\tilde{q}\in B+P,
$$
we see that $Q$ is of 2-torsion. If we write
$b_0=(1+\sigma)\tilde{q}\in B$ and $p_0=(1-\sigma)\tilde{q}\in Q$,
then we have $2\tilde{q}=b_0+p_0$. We may choose $\tilde{q}$ such
that both $b_0$ and $q_0$ are primitive. Let
$\varphi:\Z[G]\to\Lambda$ be the $G$-module homomorphism with
$\varphi(1)=\tilde{q}$. We show that $\varphi$ is injective. If
$\varphi(x+y\sigma)=0$, for some $x,y\in\Z$, then by construction,
$$
0=(x+y\sigma)(2\tilde{q})=(x+y\sigma)(b_0+p_0)=(x+y)b_0+(x-y)p_0.
$$
It follows that $x+y=0$ and $x-y=0$ and hence $x=y=0$. Let
$\Lambda'=\mathrm{coker}(\varphi)$. We next show that $\Lambda'$ is
torsion free. We will prove this by contradiction. Note that
$\Lambda'_{\mathrm{tor}}$ is isomorphic to the saturation of
$\varphi(\Z[G])$ modulo $\varphi(\Z[G])$. Assume that for some
primitive element $x+y\sigma\in\Z[G]$ we have $p\mid
\varphi(x+y\sigma)$, for some prime number $p$, then we can write
\begin{equation}\label{eq p divides}
(x+y\sigma)\tilde{q}=p\alpha
\end{equation}
for some $\alpha\in\Lambda$. If we multiply both sides of \eqref{eq
p divides} by $\sigma+1$ and note that $b_0=(\sigma+1)\tilde{q}$,
then we have $(x+y)b_0 = p\alpha$. Since $b_0$ is primitive, we see
that $p|x+y$. If we multiply both side of \eqref{eq p divides} by
$\sigma-1$ instead, we get $p\mid x-y$. Hence we get $p\mid 2x$ and
$p\mid 2y$. Since $x+y\sigma$ is primitive, we see that $x$ and $y$
are coprime. Hence $p=2$. It follows that $x$ and $y$ have the same
parity and hence both of them are odd. We write $x=2x'+1$ and
$y=2y'+1$. Hence
$$
\varphi(x+y\sigma)=2\varphi(x'+y'\sigma)+ \varphi(1+\sigma)=
2\varphi(x'+y'\sigma)+b_0.
$$
The assumption $2\mid \varphi(x+y\sigma)$ implies that $2\mid b_0$.
This is a contradiction since we choose $b_0$ to be primitive. Thus
$\Lambda'$ is torsion free.

Consider the following diagram
$$
\xymatrix{
 0\ar[r] &B'\oplus P'\ar[r] &\Lambda'\ar[r] &Q'\ar[r] &0\\
 0\ar[r] &B\oplus P\ar[r]\ar[u] & \Lambda\ar[r]\ar[u] & Q\ar[r]\ar[u] &0\\
 0\ar[r] &\Z b_0\oplus\Z p_0\ar[r]\ar[u] &\Z[G]\ar[r]\ar[u]_\varphi &\Z/2\Z\ar[r]\ar[u] &0
}
$$
The middle column implies that $\HH^1(G,\Lambda')=\HH^1(G,\Lambda)$
and $\HH^2(G,\Lambda')=\HH^2(G,\Lambda)$.

If $\dim_{\Z/2\Z}Q=0$, then the lemma holds automatically. By
induction on $\dim Q$, we may assume that
$\Lambda'=\Z[G]^{r_0-1}\oplus \Z_{+}^{r_{+}}\oplus\Z_{-}^{r_{-}}$.
Then it follows that
$$
\mathrm{Ext}_G^1(\Lambda',\Z[G])
=\mathrm{Ext}_G^1(\Z_{+},\Z[G])^{r_{+}}\oplus
\mathrm{Ext}_G^1(\Z_{-},\Z[G])^{r_-}=0.
$$
This means that the column in the middle splits and hence
$$
\Lambda=\Z[G]^{r_0}\oplus \Z_{+}^{r_{+}}\oplus\Z_{-}^{r_{-}}.
$$
This proves the lemma.
\end{proof}

\section{Prym constructions}

In this section, we fix $S$ to be a smooth projective complex
surface with $\HH_1(S,\Z)=0$. Assume that $\sigma:S\rightarrow S$ is
an automorphism with $\sigma^2=1$. In this section, we will use
$\HH^2(S)$ to denote $\HH^2(S,\Z)$; we use $h^2(S)$ to denote the
rank of $\HH^2(S,\Z)$.

\subsection{Hodge structures}
Let $M\subset \Pic(S)$ be a saturated subgroup that is also a
$G$-submodule. Let $M^\perp\subset\HH^2(S,\Z)$ be the orthogonal
complement of $M$. Then $M^\perp$ is again a $G$-module.
\begin{defn}
The \textit{Prym part} of $M^\perp$ is defined to be
$$
\mathrm{Pr}(M^\perp,\sigma)=(\sigma-1)(M^\perp),
$$
with the induced Hodge structure and an integral bilinear form
$\langle x, y\rangle=\frac{1}{2}(x\cdot y)$, where $x\cdot y$ is the
intersection pairing.
\end{defn}
Note that the bilinear form $\langle x,y\rangle$ is integral.
Indeed, we have $x=(\sigma-1)x'$ and $y=(\sigma-1)y'$ for some
$x',y'\in M^\perp$. Then we get
$$
\frac{1}{2}x\cdot y=\frac{1}{2}(\sigma x'\cdot\sigma y' +x'\cdot y'
-\sigma x'\cdot y'-x'\cdot\sigma y')=x'\cdot y' -\sigma x'\cdot
y'\in\Z.
$$
\begin{prop}\label{module structure of prym}
Assume that $\HH^1(G,M)=(\Z/2\Z)^{a_1}$ and
$\HH^2(G,M)=(\Z/2\Z)^{a_2}$. If $\sigma$ is fixed point free and
$a_1=0$, then we have
$$
 M^\perp\cong \Z[G]^{s_0}\oplus \Z_{-}^{a_2+2},
$$
for some $s_0\geq 0$. If $\sigma$ has $r>0$ isolated fixed points and
$a_2=0$, then
$$
 M^\perp\cong\Z[G]^{s_0}\oplus \Z_{+}^{a_1+r-2},
$$
for some $s_0\geq 0$.
\end{prop}
\begin{proof}
For any $G$-module $N$, we can define $N^*=\Hom_\Z(N,\Z)$, which is
again a $G$-module in a natural way. If $N$ is free as a
$\Z$-module, then we have $N^*\cong N$ as $G$-modules. By Lemma
\ref{structure lemma on G-modules}, this is reduced to
$(\Z_+)^*\cong\Z_+$, $(Z_{-})^*\cong\Z_{-}$ and $\Z[G]^*\cong\Z[G]$,
which are easy to verify. We consider the following exact sequence
$$
\xymatrix{
 0\ar[r] &M\ar[r] &\HH^2(S,\Z)\ar[r] &(M^\perp)^*\ar[r] &0.
}
$$
If $\sigma$ is fixed point free, then Theorem \ref{thm on surface
with involution fixed point free} shows that $\HH^2(G,\HH^2(S))=0$
and $\HH^1(G,\HH^2(S))=(\Z/2\Z)^2$. If furthermore $a_1=0$, then
from the above short exact sequence we get the following exact
sequence
$$
\xymatrix{
 0\ar[r] &\HH^1(G,\HH^2(S))\ar[r] &\HH^1(G,(M^\perp)^*)\ar[r]
 &\HH^2(G,M)\ar[r] &0
 }
$$
and $\HH^2(G,(M^\perp)^*)=0$. Since $M^\perp\cong (M^\perp)^*$, we
get
$$
\HH^1(G,M^\perp)=(\Z/2\Z)^{a_2+2}\text{ and } \HH^2(G,M^\perp)=0.
$$
Then the conclusion follows from Lemma \ref{structure lemma on
G-modules}.

If $\sigma$ has $r>0$ fixed points, then Theorem \ref{thm on surface
with involution} shows that $\HH^1(G,\HH^2(S))=0$ and
$\HH^2(G,\HH^2(S))=(\Z/2\Z)^{r-2}$. If furthermore $a_2=0$, then the
associated long exact sequence becomes
$$
\xymatrix{
 0\ar[r] &\HH^2(G,\HH^2(S))\ar[r] &\HH^2(G,(M^\perp)^*)\ar[r]
 &\HH^3(G,M)\ar[r] &0
}
$$
and $\HH^1(G,(M^\perp)^*)=0$. From this we get
$$
\HH^1(G,M^\perp)=0 \text{ and } \HH^2(G,M^\perp)=(\Z/2\Z)^{a_1+r-2}.
$$
Then the conclusion follows again from Lemma \ref{structure lemma on
G-modules}.
\end{proof}

\begin{cor}\label{rank of prym}
If $\sigma$ is fixed point free and $a_1=0$, then
$$
\rk \mathrm{Pr}(M^\perp,\sigma)=\frac{h^2(S)-\rk M+a_2}{2}+1.
$$
If $\sigma$ has $r>0$ fixed points and $a_2=0$, then
$$
\mathrm{Pr}(M^\perp,\sigma)=\{\alpha\in
M^\perp:\sigma(\alpha)=-\alpha\}
$$
and
$$
\rk \mathrm{Pr}(M^\perp,\sigma)=\frac{h^2(S)-\rk M-a_1-r}{2}+1.
$$
\end{cor}
\begin{proof}
If $\sigma$ is fixed point free and $a_1=0$, then
$2s_0+a_2+2=h^2(S)-\rk M$. Hence we get $s_0=\frac{h^2(S)-\rk
M-a_2}{2}-1$. By definition, we get
$$
\rk \mathrm{Pr}(M^\perp,\sigma)=s_0+a_2+2=\frac{h^2(S)-\rk
M+a_2}{2}+1.
$$
The remaining case is proved similarly.
\end{proof}
We fix some notations. For any $G$-module $N$, we write
$$
N^{\sigma=1}=N^G=\{x\in N:\sigma(x)=x\} \text{ and }
N^{\sigma=-1}=\{x\in N:\sigma(x)=-x\}.
$$
Let $N$ be a free $\Z$-module with a symmetric bilinear form, then
$\det(N)$ is defined to be the determinant of a matrix
representation of the bilinear form. For example, when $\Sigma$ is a
smooth projective surface over $\C$, the cup-product (or
intersection form) defines a symmetric bilinear form on
$\Lambda=\HH^2(\Sigma,\Z)/\mathrm{tor}$. The Poincar\'e duality
implies that $\det(\Lambda)=\pm 1$.

\begin{prop}\label{determinant of prym}
Let $M\subset \Pic(S)$ be a saturated submodule as above such that
the intersection form restricted to $M$ is nondegenerate and
$2\nmid\det(M)$. Assume that $\HH^1(G,M)=(\Z/2\Z)^{a_1}$ and
$\HH^2(G,M)=(\Z/2\Z)^{a_2}$. Let $M\hookrightarrow M^*$ be the
natural induced inclusion with quotient denoted by $Q_M$. Let
$q'=\det(M^{\sigma=-1})$ and $q=|(Q_M)^{\sigma=-1}|$. If $\sigma$ is
fixed point free and $a_1=0$, then
$$
\det(\mathrm{Pr}(M^\perp,\sigma),\langle\,,\rangle)=\pm
2^{a+2}\frac{q^2}{q'},\quad a=\frac{\rk M+3a_2}{2}.
$$
If $\sigma$ has $r>0$ fixed points and $a_2=0$, then
$$
\det(\mathrm{Pr}(M^\perp,\sigma),\langle\,,\rangle)=\pm
2^{b}\frac{q^2}{q'},\quad b=\frac{\rk M+a_1}{2}.
$$
\end{prop}

\begin{proof}
We first prove the following

\textit{Claim 1}: Let $r_0$ be the number of copies of $Z[G]$ in
$\HH^2(S,\Z)$, then $\det(\HH^2(S)^{\sigma=-1})=\pm 2^{r_0}$.

Consider the following exact sequence
$$
\xymatrix{
 0\ar[r] &\HH^2(S)^{\sigma=1}\oplus\HH^2(S)^{\sigma=-1}\ar[r]
 &\HH^2(S)\ar[r] &(\Z/2\Z)^{r_0}\ar[r] &0.
}
$$
By Lemma \ref{discriminant of sub module}, we get
\begin{equation}\label{equation on product}
\det(\HH^2(S)^{\sigma=1})\det(\HH^2(S)^{\sigma=-1})=\pm 2^{2r_0}.
\end{equation}
Let $Y=S/G$ be the quotient of $S$ by the involution $\sigma$ and
$\pi:S\rightarrow Y$ the quotient map. We use $\HH^2(Y)$ to denote
$\HH^2(Y,\Z)$ modulo torsion. If $\sigma$ is fixed point free, then
$Y$ is smooth and there is an isomorphism $\pi^*\HH^2(Y)=
\HH^2(S)^G$. Hence we have
$$
\det{\HH^2(S)^G}=\pm 2^{h^2(Y)}=\pm 2^{r_0}.
$$
By combining this with the identity above, we get
$$
\det(\HH^2(S)^{\sigma=-1})=\pm 2^{r_0}.
$$
If $\sigma$ has $r>0$ fixed points, the $Y$ has $r$ ordinary double
points. Let $f:\tilde{Y}\rightarrow Y$ be the minimal resolution of
$Y$. Let $E_i\subset \tilde{Y}$, $i=1,\ldots,r$, be the exceptional
divisors. As in Remark \ref{remark on Leray}, we see that
$f^*\HH^2(Y)\subset\HH^2(\tilde{Y})$ is simply the orthogonal
complement of $\{E_1,\ldots,E_r\}$. Let $E\subset\HH^2(\tilde{Y})$
be the subgroup $\oplus_{i=1}^r \Z E_i$. From Remark \ref{remark on
Leray}, we see that the homomorphism $\HH^2(\tilde{Y})\rightarrow
\Z^r$, defined by intersecting with $E_i$'s, has cokernel
$N=\Z/2\Z$. Note that $E^\perp = f^*\HH^2(Y)$, then the Leray
spectral sequence in Remark \ref{remark on Leray} gives rise to the
following exact sequence
$$
\xymatrix{
 0\ar[r] &E^\perp\ar[r] &\HH^2(\tilde{Y})\ar[r]
 &\oplus\HH^2(E_i,\Z)\ar[r] &N=\Z/2\Z\ar[r] &0 }
$$
From this exact sequence, we can construct the following diagram
$$
\xymatrix{
 0\ar[r] &\frac{\HH^2(\tilde{Y})}{E\oplus E^\perp} \ar[r]
 &(\Z/2\Z)^r\ar[r] &\Z/2\Z\ar[r] &0\\
 0\ar[r] &\frac{\HH^2(\tilde{Y})}{E^\perp}\ar[r]\ar[u]
 &\bigoplus_{i=1}^{r}\HH^2(E_i,\Z)\ar[r]\ar[u]
 &\Z/2\Z\ar[r]\ar@{=}[u] &0\\
 &E\ar@{=}[r]\ar[u] &E\ar[u] & &
}
$$
It follows that $\frac{\HH^2(\tilde{Y})}{E\oplus
E^\perp}=(\Z/2\Z)^{r-1}$, and hence we get the following exact
sequence
$$
\xymatrix{
 0\ar[r] &E^\perp\oplus E\ar[r] &\HH^2(\tilde{Y})\ar[r]
 &(\Z/2\Z)^{r-1}\ar[r] &0}
$$
From this and Lemma \ref{discriminant of sub module}, we deduce that
$\det(E)\det(E^\perp)=\pm 2^{2r-2}$. Since $\det(E)= (-2)^r$, it
follows that
$$
\det(\HH^2(Y))=\det(f^*\HH^2(Y))=\det(E^\perp)=\pm 2^{r-2}.
$$
Note that $\det(\pi^*\HH^2(Y))=2^{h^2(Y)}\det(\HH^2(F))$ since
$\deg(\pi)=2$ and also that
$$
\mathrm{coker}\{\pi^*\HH^2(Y)\hookrightarrow\HH^2(S)^G\}=(\Z/2\Z)^{r-2}.
$$
which follows from (iv) of Theorem \ref{thm on surface with
involution}. By Lemma \ref{discriminant of sub module}, we have
$$
2^{h^2(Y)}\det(\HH^2(Y))=\det(\pi^*\HH^2(Y))=2^{2r-4}\det(\HH^2(S)^G).
$$
Note that $h^2(Y)=r_0+r-2$. Hence $\det(\HH^2(S)^G)=\pm 2^{r_0}$.
Then it follows from \eqref{equation on product} that
$\det(\HH^2(S)^{\sigma=-1})=\pm 2^{r_0}$. This proves Claim 1.

Now consider the following short exact sequence
$$
\xymatrix{
 0\ar[r] &M\oplus M^\perp\ar[r] &\HH^2(S)\ar[r] &Q_M\ar[r] &0.
}
$$
By taking the anti-invariant part (meaning the part on which
$\sigma=-1$), we have a long exact sequence
$$
\xymatrix{
 0\ar[r] &M^{\sigma=-1}\oplus (M^\perp)^{\sigma=-1}\ar[r]
 &\HH^2(S)^{\sigma=-1}\ar[r] &\\
 (Q_M)^{\sigma=-1}\ar[r]^{\delta\quad}
 &\HH^2(G,M\oplus M^\perp)\ar[r] &\cdots &
}
$$
Note that as a $G$-module, $Q_M$ is canonically isomorphic to
$M^*/M$ where $M\subset M^*$ is induced by the intersection form.
The fact that $2\nmid \det(M)$ implies that $Q_M$ has no 2-torsion.
Since $\HH^2(G,M\oplus M^\perp)$ is always of 2-torsion by Lemma
\ref{lem 2-torsion}, we conclude that $\delta=0$. By Lemma
\ref{discriminant of sub module}, we have
$$
\det(M^{\sigma=-1})\det((M^\perp)^{\sigma=-1})=q^2\det(\HH^2(S)^{\sigma=-1}).
$$
or equivalently,
$$
\det((M^\perp)^{\sigma=-1})=\frac{q^2}{q'}\det(\HH^2(S)^{\sigma=-1}).
$$
If $\sigma$ is fixed point free, then the quotient of
$\mathrm{Pr}(M^\perp,\sigma)\subset (M^\perp)^{\sigma=-1}$ is
$(\Z/2\Z)^{a_2+2}$. Then we have
\begin{align*}
 \det(\mathrm{Pr}(M^\perp,\sigma),\langle\,,\rangle)
 &=\frac{1}{2^s}\det(\mathrm{Pr}(M^\perp,\sigma),\cdot),\quad s=\rk\mathrm{Pr}(M^\perp,\sigma)\\
 &=2^{2a_2-s+4}\det(M^\perp)^{\sigma=-1}\\
 &=\pm 2^{a+2}\frac{q^2}{q'},
\end{align*}
where $a=\frac{\rk M+3a_2}{2}$. In the case when $\sigma$ has at
least one fixed point, we have
\begin{align*}
 \det(\mathrm{Pr}(M^\perp,\sigma),\langle\,,\rangle)
 &=\frac{1}{2^s}\det(\mathrm{Pr}(M^\perp,\sigma),\cdot),\quad s=\rk\mathrm{Pr}(M^\perp,\sigma)\\
 &=2^{-s}\det(M^\perp)^{\sigma=-1}\\
 &=\pm 2^{b}\frac{q^2}{q'},
\end{align*}
where $b=\frac{\rk M+a_1}{2}$.
\end{proof}

\begin{lem}\label{discriminant of sub module}
Let $N$ be a free $\Z$-module of finite rank which is equipped with
an integral symmetric bilinear form and $N'\subset N$ a submodule of
same rank. Then we have $\det(N')=|N/N'|^2\det(N)$.
\end{lem}
\begin{proof}
Let $\mathbf{v}=(v_1,\ldots,v_n)$ be an integral basis of $N$, where
$n=\rk N$. Let $\mathbf{v}'=(v'_1,\ldots,v'_n)$ be an integral basis
of $N'$. Then there is a matrix $A$ such that
$\mathbf{v}'=\mathbf{v}A$ and $|N/N'|=|\det(A)|$. Let $T$ be the
matrix representing the bilinear form of $N$ under the basis
$\mathbf{v}$. Then the matrix representation of the bilinear form on
$N'$ with respect to $\mathbf{v}'$ is $T'=A^t\, T A$. Hence by
definition, we have
$$
\det(N')=\det(T')=\det(A^t\,T A)=\det(A)^2\det(T)=|N/N'|^2\det(N).
$$
This proves the lemma.
\end{proof}

\subsection{Chow group and Brauer group}
We use $\mathrm{A}_0(S)$ to be the Chow group of zero cycles of
degree 0. Then $\sigma$ also acts on $\mathrm{A}_0(S)$. Let
$\mathrm{Br}(S)$ be the Brauer group of the surface $S$. Again
$\sigma$ induces an involution on $\mathrm{Br}(S)$.
\begin{defn}
The \textit{Prym part} of $\mathrm{A}_0(S)$ is
$$
\mathrm{Pr}(\mathrm{A}_0(S),\sigma)=(\sigma-1)\mathrm{A}_0(S).
$$
The \textit{Prym part} of $\mathrm{Br}(S)$ is
$$
\mathrm{Pr}(\mathrm{Br}(S),\sigma)=(\sigma-1)\mathrm{Br}(S).
$$
\end{defn}
\begin{rmk}
Under the assumption that $\HH_1(S,\Z)=0$, By Roitman's theorem
\cite{roitman}, the group $\mathrm{A}_0(S)$ is always torsion free
and uniquely divisible. Then $\mathrm{Pr}(\mathrm{A}_0(S),\sigma)$
is also torsion free and uniquely divisible. This further implies
that
$$
\mathrm{Pr}(\mathrm{A}_0(S),\sigma)=\mathrm{A}_0(S)^{\sigma=-1}.
$$
However, the same identity for Prym part of Brauer group does not
hold since $\mathrm{Br}(S)$ is torsion.
\end{rmk}

Consider the long exact sequence associated to the exponential
sequence on $S$, we have
$$
\xymatrix{
 0\ar[r] &\mathrm{Hdg}^2(S)\ar[r] &\HH^2(S,\Z)\ar[r] &\HH^2(S,\calO_S)\ar[r] &\HH^2(S,\calO_S^*)\ar[r] &0.
}
$$
Since $\mathrm{Br}(S)=\HH^2(S,\calO_S^*)_\mathrm{tor}$ (see
\cite[Cor. 2.2]{grothemdieck2}), we get
$$
\mathrm{Br}(S)=T^2(S)\otimes\Q/\Z=\Hom(\HH^2(S,\Z)_{\mathrm{tr}},\Q/\Z).
$$
Let $M\subset\mathrm{Hdg}^2(S)=\Pic(S)$ be saturated and
nondegenerate as above. We still set $Q_M=\HH^2(S,\Z)/(M\oplus
M^\perp)$. Let $\mathrm{Hdg}(Q_M)$ be the image of
$\mathrm{Hdg}^2(S)$ in $Q_M$. Set $T(Q_M)=Q_M/\mathrm{Hdg}(Q_M)$.
Then we have the following diagram
\begin{equation}\label{big diagram 1}
\xymatrix{
 &0 &0 &0 &\\
 0\ar[r] &\mathrm{Hdg}(Q_M)\ar[r]\ar[u] &Q_M\ar[r]\ar[u] &T(Q_M)\ar[r]\ar[u] &0\\
 0\ar[r] &\mathrm{Hdg}^2(S)\ar[r]\ar[u] &\HH^2(S,\Z)\ar[r]\ar[u] &T^2(S)\ar[r]\ar[u] &0\\
 0\ar[r] &\mathrm{Hdg}(M^\perp)\oplus M\ar[r]\ar[u] &M^\perp\oplus M\ar[r]\ar[u] &T(M^\perp)\ar[r]\ar[u] &0\\
 &0\ar[u] &0\ar[u] &0\ar[u] &
}
\end{equation}
If we take the $\sigma=-1$ parts and assume that $2\nmid\det(M)$,
then we get the following short exact sequence
\begin{equation}\label{sequence 1}
\xymatrix{
 0\ar[r] &T(M^\perp)^{\sigma=-1}\ar[r] &T^2(S)^{\sigma=-1}\ar[r] &T(Q_M)^{\sigma=-1}\ar[r] &0.
}
\end{equation}
The surjectivity on the right hand side follows from the fact that
$T(Q_M)$ has no 2-torsion and that $\HH^2(G,T(M^\perp))$ is of
2-torsion (Lemma \ref{lem 2-torsion}).
\begin{lem}\label{lem anti-invariant}
Let $N$ be a $G$-module which is free of finite rank over $\Z$. Then
$$
N^{\sigma=-1}\otimes\Q/\Z=\mathrm{Pr}(N\otimes\Q/\Z,\sigma)=(\sigma-1)(N\otimes\Q/\Z).
$$
\end{lem}
\begin{proof}
By Lemma \ref{structure lemma on G-modules}, we only need to check
for $N=\Z_{+}$, $N=\Z_{-}$ and $N=\Z[G]$, which are easy to verify.
\end{proof}
The above lemma leads to the following
\begin{prop}\label{exact sequence on Prym of Brauer}
Let $M\subset\Pic(S)$ be a $G$-module that is saturated and
nondegenerate. Assume that $2\nmid\det(M)$. Then there is a short
exact sequence
$$
\xymatrix{
 0\ar[r] &T(Q_M)^{\sigma=-1}\ar[r] &T(M^\perp)^{\sigma=-1}\otimes\Q/\Z\ar[r] &\mathrm{Pr}(\mathrm{Br}(S),\sigma)\ar[r] &0.
}
$$
We also have another short exact sequence
$$
\xymatrix{
 0\ar[r] &K\ar[r] &T((M^\perp)^{\sigma=-1})\otimes\Q/\Z\ar[r] &T(M^\perp)^{\sigma=-1}\otimes\Q/\Z\ar[r] &0,
}
$$
where
$K=\mathrm{coker}\{\HH^1(G,\HH^2(S,\Z))\rightarrow\HH^1(G,T^2(S))\}$.
\end{prop}
\begin{cor}\label{exact sequence on Prym of Brauer II}
There is a short exact sequence
$$
\xymatrix{
 0\ar[r] &K\oplus T(Q_M)^{\sigma=-1}\ar[r] &T((M^\perp)^{\sigma=-1})\otimes\Q/\Z\ar[r] &\mathrm{Pr}(\mathrm{Br}(S),\sigma)\ar[r] &0.
}
$$
\end{cor}
\begin{proof}[Proof of Proposition \ref{exact sequence on Prym of Brauer}]
By Lemma \ref{lem anti-invariant}, we have
$$
T^2(S)^{\sigma=-1}\otimes \Q/\Z =\mathrm{Pr}(\mathrm{Br}(S),\sigma).
$$
We tensor sequence \eqref{sequence 1} with $\Q/\Z$ and get
$$
\xymatrix{
 0\ar[r] &\mathrm{Tor}_1(T(Q_M)^{\sigma=-1},\Q/\Z)\ar[r] &T(M^\perp)^{\sigma=-1}\otimes\Q/\Z\ar[r] &\mathrm{Pr}(\mathrm{Br}(S),\sigma)\ar[r] &0.
}
$$
Since $\mathrm{Tor}_1(T(Q_M)^{\sigma=-1},\Q/\Z)\cong
T(Q_M)^{\sigma=-1}$, the first exact sequence follows. Consider the
short exact sequence
$$
\xymatrix{
 0\ar[r] &\mathrm{Hdg}(M^\perp)\ar[r] &M^\perp\ar[r] &T(M^\perp)\ar[r] &0.
}
$$
By taking the $\sigma=-1$ part, we get
$$
\xymatrix{
 0\ar[r] &\mathrm{Hdg}(M^\perp)^{\sigma=-1}\ar[r] &(M^\perp)^{\sigma=-1}\ar[r] &T(M^\perp)^{\sigma=-1}\ar[r] &K'\ar[r] &0,
}
$$
where $K'=\ker\{\HH^2(G,\mathrm{Hdg}(M^\perp))\rightarrow
\HH^2(G,M^\perp)\}$. In the diagram \eqref{big diagram 1}, we note
that $Q_M$ (and hence $\mathrm{Hdg}(Q_M)$ and $T(Q_M)$) is free of
$2$-torsion. This implies $\HH^i(G,Q_M)=0$ (and hence
$\HH^i(G,\mathrm{Hdg}(Q_M))=0$ and $\HH^i(G,T(Q_M))=0$) for all
$i>0$. The left and middle columns of \eqref{big diagram 1} show
that
$$
\HH^2(G,M\oplus\mathrm{Hdg}(M^\perp))\cong\HH^2(G,\mathrm{Hdg}^2(S))
$$
and
$$
\HH^2(G,M^\perp\oplus M)\cong\HH^2(G,\HH^2(S,\Z)).
$$
Hence we have
\begin{align*}
 K' &=\ker\{\HH^2(G,\mathrm{Hdg}(M^\perp)\oplus M)\rightarrow \HH^2(G,M^\perp\oplus M)\}\\
    &=\ker\{\HH^2(G,\mathrm{Alg}^2(S))\rightarrow \HH^2(G,\HH^2(S,\Z))\}\\
    &=\mathrm{coker}\{\HH^1(G,\HH^2(S,\Z))\rightarrow \HH^1(G,T^2(S))\}\\
    &=K.
\end{align*}
This gives the following short exact sequence
$$
\xymatrix{
 0\ar[r] &T((M^\perp)^{\sigma=-1})\ar[r] &T(M^\perp)^{\sigma=-1}\ar[r] &K\ar[r] &0.
}
$$
By tensoring with $\Q/\Z$, we get the second exact sequence.
\end{proof}

\subsection{Applications to Fourfolds}
On a variety $Y$, the Deligne complex, $\Z(k)_\mathcal{D}$, is the complex
$$
\xymatrix{
 \Z(k)\ar[r] &\calO_Y\ar[r] &\Omega_Y^1\ar[r] &\cdots\ar[r]
 &\Omega_Y^{k-1}
}
$$
where $\Z(k)=(2\pi i)^k\Z$ is in degree 0. The Deligne cohomology
groups are defined to be the hypercohomology of the Deligne complex.
To be more precise,
$$
\HH^j_\mathcal{D}(Y,\Z(k))=\mathbb{H}^j(Y,\Z(k)_\mathcal{D}).
$$
For small values of $j$ and $k$, these groups are easy to
understand. For example, the exponential sequence shows that
$\Z(1)_\mathcal{D}$ is quasi-isomorphic to $\calO_{Y^{an}}^*[-1]$.
Hence we know that $\HH^2_{\mathcal{D}}(Y,\Z(1))\cong\Pic(Y)$ and
$\HH^3_\mathcal{D}(Y,\Z(1))_{\mathrm{tor}}\cong\mathrm{Br}(Y)$.
\begin{defn}\label{defn of second Brauer group}
Let $Y$ be a smooth projective variety. We define
$$
\mathrm{Br}_2(Y)=\HH_\mathcal{D}^5(Y,\Z(2))_\mathrm{tor}.
$$
\end{defn}
\begin{rmk}
If $H^i(Y,\calO_Y)=0$, for all $i>0$, then the group
$\mathrm{Br}_2(Y)$ is quite similar to the Brauer group of a
surface. If we further assume that $\HH^5(Y,\Z)=0$, then the long
exact sequence associated to
$$
\xymatrix{
 0\ar[r]
  &\Omega_Y^1[-2]\ar[r]
  &\Z_\mathcal{D}(2)\ar[r] &\{\Z(2)\rightarrow\calO_Y\}\ar[r]
  &0
}
$$
gives the following exact sequence
$$
\xymatrix{
 0\ar[r] &\mathrm{Hdg}^4(Y)\ar[r] &\HH^4(Y,\Z)\ar[r]
 &\HH^{1,3}(Y)\ar[r] &\HH^5_\mathcal{D}(Y,\Z(2))\ar[r] &0.
}
$$
Here we use the fact that
$\HH^4(Y,\Z)\cong\HH^4(Y,\{\Z(2)\to\calO_Y\})$, which is a
consequence of the vanishing of $\HH^i(Y,\calO_Y)$, $i>0$. The above
exact sequence implies that $\mathrm{Br}_2(Y)=T^4(Y)\otimes
\Q/\Z=\Hom(\HH^4(Y,\Z)_\mathrm{tr},\Q/\Z)$.
\end{rmk}

Let $X$ be a smooth projective fourfold with $\HH^i(X,\calO_X)=0$,
$\forall i>0$. Assume that $S$ parameterizes a family of curves on
$X$ given by
$$
\xymatrix{
 \mathscr{C}\ar[d]_p\ar[r]^q &X\\
 S &
}
$$
Let $\Phi=p_*q^*$ be the Abel-Jacobi homomorphism defined on the
cohomology groups, the Chow groups and the Brauer groups. To be more
precise we have
\begin{align*}
 \Phi: &\HH^4(X,\Z)\rightarrow \HH^2(S,\Z), \\
 \Phi: &\mathrm{CH}_1(X)\rightarrow \mathrm{CH}_0(S),\\
 \Phi: &\mathrm{Br}_2(X)\rightarrow \mathrm{Br}(S).
\end{align*}
The cylinder homomorphism $\Psi=q_*p^*$ is defined similarly. Let
$N$ be the saturation of the subgroup
$\Sym^2(\HH^2(X,\Z))\subset\HH^4(X,\Z)$ of cohomology classes coming
from $\HH^2(X,\Z)$. We define the primitive cohomology of $X$ to be
$$
 \HH^4(X,\Z)_\mathrm{prim}=N^\perp\subset\HH^4(X,\Z).
$$
Let $\mathrm{A}_1(X)=\CH_1(X)_\mathrm{hom}$ be the Chow group of
homologically trivial 1-cycles on $X$. In our case,
$\mathrm{A}_1(X)$ consists all elements of $\alpha\in\CH_1(X)$ such
that the intersection number $\alpha\cdot D$ is zero for all
divisors $D$ on $X$. Let $M\subset\Pic(S)$ be a $G$-submodule such
that
$$
\Phi(N)\subset M\text{ and } \Psi(M)\subset N.
$$
Then the homomorphism $\Phi$ induces
\begin{align*}
 \Phi: &\HH^4(X,\Z)_\mathrm{prim}\rightarrow M^\perp, \\
 \Phi: &\mathrm{A}_1(X)\rightarrow \mathrm{A}_0(S),
\end{align*}
where $\mathrm{A}_0(S)$ is the Chow group of 0-cycles of degree
zero. Similarly, we have
\begin{align*}
\Psi:&M^\perp\rightarrow \HH^4(X,\Z)_\mathrm{prim},\\
\Psi:&\mathrm{A}_0(S)\rightarrow \mathrm{A}_1(X).
\end{align*}

\begin{thm}\label{thm on prym construction}
Let $S$ be a smooth projective surface with $\HH_1(S,\Z)=0$ and
$\sigma:S\rightarrow S$ an involution with isolated fixed points.
Let $X$ be a 4-dimensional smooth projective variety with
$\HH^i(X,\calO_X)=0$, for all $i>0$ and $\HH^5(X,\Z)=0$. Assume that
$S$ parameterizes a family, $p:\mathscr{C}\rightarrow S$, of curves
on $X$ with $q:\mathscr{C}\rightarrow X$ being the natural map. Set
$\Phi=p_*q^*$ and $\Psi=q_*p^*$ as above, which are defined on the
cohomology groups, the Chow groups and the Brauer group. Let
$N\subset\HH^4(X,\Z)$ be the saturation of the image of
$\Sym^2(\HH^2(X,\Z))$. Let $M\subset\Pic(S)$ be a saturated and
nondegenerate $G$-submodule such that $\Phi(N)\subset M$ and
$\Psi(M)\subset N$.

(i) Assume that the following three conditions hold,
\begin{align*}
 &\text{(a). }\Psi\circ\Phi =-2: \HH^4(X,\Z)_\mathrm{prim} \rightarrow
 \HH^4(X,\Z)_\mathrm{prim};\\
 &\text{(b). }\Phi\circ\Psi=\sigma-1:  M^\perp\rightarrow M^\perp;\\
 &\text{(c). }\HH^1(G,M^\perp)=0.
\end{align*}
Then $\Phi$ induces an isomorphism
$$
\Phi:\HH^4(X,\Z)_{\mathrm{prim}}\rightarrow
\mathrm{Pr}(M^\perp,\sigma)(-1),
$$
which respects the bilinear forms of the two sides, namely
$$
 \langle\Phi(x),\Phi(y)\rangle=-(x\cdot y)_X,\quad\forall
 x,y\in\HH^4(X,\Z)_\mathrm{prim}.
$$
The condition (c) above is fulfilled if $\sigma$ has fixed points
and $\HH^1(G,M)=0$.

(ii) Assume that the following two conditions hold,
\begin{align*}
 &\text{(a). } \Psi\circ\Phi=-2:\mathrm{A}_1(X)\rightarrow
 \mathrm{A}_1(X);\\
 &\text{(b). } \Phi\circ\Psi=\sigma-1:\mathrm{A}_0(S)\rightarrow
 \mathrm{A}_0(S).
\end{align*}
Then there is a canonical isomorphism
$$
\mathrm{A}_1(X)\cong\mathrm{Pr}(\mathrm{A}_0(S),\sigma)\oplus(\text{2-torsion}),
$$
such that
$\Phi:\mathrm{A}_1(X)\rightarrow\mathrm{Pr}(\mathrm{A}_0(S),\sigma)$
is simply the projection.

(iii) Let the assumptions be the same as in (i). If
$2\nmid\det(M)\det(N)$, then for ``Brauer" groups, we have a short
exact sequence
$$
\xymatrix{
 0\ar[r] &K\ar[r] &\mathrm{Br}_2(X)\ar[r]^{\Phi\quad}
 &\mathrm{Pr}(\mathrm{Br}(S),\sigma)\ar[r] &0,
}
$$
where
$K=\mathrm{coker}\{\HH^1(G,\HH^2(S,\Z))\rightarrow\HH^1(G,T^2(S))\}$.
\end{thm}

\begin{proof}
We note that the assumption $\HH^5(X,\Z)=0$ implies that
$\HH_3(X,\Z)=0$ by the Poincar\'e duality and hence $\HH^4(X,\Z)$ is
torsion-free by the universal coefficient theorem for cohomology. To
prove (i), we first show that
$\Phi:\HH^4(X,\Z)_\mathrm{prim}\rightarrow M^\perp$ is injective
with $\mathrm{Im}(\Phi)\subset \mathrm{Pr}(M^\perp,\sigma)$. For any
element $\alpha\in\HH^4(X,\Z)_\mathrm{prim}$, set
$\fa=\Phi(\alpha)\in M^\perp$. If $\fa=0$, then
$$
-2\alpha=\Psi(\Phi(\alpha))=\Psi(\fa)=0.
$$
This implies $\alpha=0$ since $\HH^4(X,\Z)$ is torsion free. This
shows that $\Phi$ is injective. We also have
$$
(\sigma-1)\fa=\Phi\circ\Psi(\Phi(\alpha))=\Phi(-2\alpha)=-2\fa,
$$
which implies $\sigma(\fa)=-\fa$, namely $\fa\in
(M^\perp)^{\sigma=-1}$. Since $\HH^1(G,M^\perp)=0$, we know that
$\fa\in\mathrm{Pr}(M^\perp,\sigma)$. We next show that
$\mathrm{Im}(\Phi)=\Pr(M^\perp,\sigma)$. To do this we only need to
show that $\Psi:M^\perp\rightarrow\HH^4(X,\Z)_\mathrm{prim}$ is
surjective. Indeed for any $\alpha\in\HH^4(X,\Z)_\mathrm{prim}$, we
know that $\fa=\Phi(\alpha)$ can be written as
$\fa=(\sigma-1)\fa_0$, for some $\fa_0\in M^\perp$. Let
$\alpha'=\Psi(\fa_0)$, then
$$
\Phi(\alpha')=\Phi\circ\Psi(\fa_0)=(\sigma-1)\fa_0=\fa=\Phi(\alpha).
$$
Since $\Phi$ is injective, this forces $\alpha=\alpha'=\Psi(\fa_0)$.
Hence $\Psi$ is surjective. We still need to check the compatibility
of the bilinear forms. Let $x,y\in\HH^4(X,\Z)_\mathrm{prim}$, then
$$
\langle
\Phi(x),\Phi(y)\rangle=\frac{1}{2}(\Phi(x)\cdot\Phi(y))_S=\frac{1}{2}(x\cdot\Psi\Phi(y))_X
=-(x\cdot y)_X.
$$

Let $\beta\in\mathrm{A}_1(X)$ be arbitrary and set
$\fb=\Phi(\beta)\in\mathrm{A}_0(S)$. Then, in the same way as for
cohomology, we can show that $\sigma(\fb)=-\fb$. If $\fb=0$, then
$-2\beta=\Psi(\fb)=0$; namely $\beta$ is 2-torsion. Conversely, if
$\beta$ is 2-torsion, then $\Phi(\beta)=0$ since $\mathrm{A}_0(S)$
is torsion free by a theorem of Roitman \cite{roitman}. This implies
that
$$
\ker(\Phi)=\mathrm{A}_1(X)[2]=\{\beta\in\mathrm{A}_1(X):2\beta=0\}.
$$
Since $\mathrm{A}_0(S)$ is uniquely divisible, the element
$\fb_0=-\frac{1}{2}\fb\in\mathrm{A}_0(S)$ is well-defined and
$\sigma(\fb_0)=-\fb_0$. Let $\beta'=\Psi(\fb_0)\in\mathrm{A}_1(X)$.
Then
$$
\Phi(\beta')=(\sigma-1)\fb_0=-2\fb_0=\fb=\Phi(\beta).
$$
This implies that $\beta-\beta'\in\ker(\Phi)$ and hence $\Psi$ is
surjective modulo 2-torsion. Thus we conclude that
$\Phi:\mathrm{A}_1(X)\rightarrow
\mathrm{Pr}(\mathrm{A}_0(S),\sigma)$ is surjective with kernel equal
to $\mathrm{A}_0(S)[2]$. The image of $\Psi$ gives the natural
splitting. This proves (ii).

We start to prove (iii). Consider the following short exact sequence
$$
\xymatrix{
 0\ar[r] &N\oplus N^\perp\ar[r] &\HH^4(X,\Z)\ar[r] &Q_N\ar[r] &0.
}
$$
Note that all classes in $N$ are hodge classes and hence $T(N\oplus
N^\perp)=T(N^\perp)$. Then we can deduce the following exact
sequence from the one above,
$$
\xymatrix{
0\ar[r] &T(N^\perp)\ar[r] &T^4(X)\ar[r] &T(Q_N)\ar[r] &0,
}
$$
where $T(Q_N)$ is the quotient of $Q_N$ be the image of Hodge
classes. By tensoring with $\Q/\Z$ and note that
$\mathrm{Tor}_1(T(Q_N),\Q/\Z)\cong T(Q_N)$ and
$\mathrm{Br}_2(X)=T^4(X)\otimes\Q/\Z$. we get a short exact sequence
relating $\mathrm{Br}_2(X)$ to the group
$T(\HH^4(X,\Z)_\mathrm{prim})$, namely
$$
\xymatrix{
 0\ar[r] &T(Q_N)\ar[r] &T(\HH^4(X,\Z)_\mathrm{prim})\otimes
 \Q/\Z\ar[r] &\mathrm{Br}_2(X)\ar[r] &0.
}
$$
Here $Q_N=N^*/N=\HH^4(X,\Z)/(N\oplus\HH^4(X,\Z)_\mathrm{prim})$. Let
$\mathrm{Hdg}(Q_N)$ be the image of $\mathrm{Hdg}^4(X)$ in $Q_N$,
then $T(Q_N)=Q_N/\mathrm{Hdg}(Q_N)$. Note that by assumption, $Q_N$
has no 2-torsion and hence neither does $T(Q_N)$. The above short exact
sequence together with the one obtained in Corollary \ref{exact
sequence on Prym of Brauer II}, we get the following diagram
$$
\xymatrix{
 0\ar[r] &K\oplus T(Q_M)^{\sigma=-1}\ar[r]
 &T((M^\perp)^{\sigma=-1})\otimes\Q/\Z\ar[r]
 &\mathrm{Pr}(\mathrm{Br}(S),\sigma)\ar[r] &0\\
 0\ar[r] &T(Q_N)\ar[r]\ar[u]_\phi &T(\HH^4(X,\Z)_\mathrm{prim})\otimes
 \Q/\Z\ar[r]\ar[u]_{\cong} &\mathrm{Br}_2(X)\ar[r]\ar[u]_\Phi &0
}
$$
where the middle column is isomorphism by (i). The assumptions imply that
$\Psi\circ\Phi=-2:\mathrm{Br}_2(X)\rightarrow \mathrm{Br}_2(X)$.
This implies $\ker(\Phi)$ is of 2-torsion. The snake lemma shows
that $\ker(\Phi)$ is isomorphic to $\mathrm{coker}(\phi)$. Since
$T(Q_M)$ and $T(Q_N)$ are all free of 2-torsion, we conclude that
$\phi:T(Q_N)\rightarrow T(Q_M)^{\sigma=-1}$ is isomorphism and
$\ker(\Phi)\cong K$.
\end{proof}

\section{Cubic fourfolds}

In this section, we fix $X\subset\PP^5_{\C}$ to be a smooth cubic
fourfold. Let $F=F(X)$ be the variety of lines on $X$. One naturally
embeds $F$ into the Grassmannian $G(2,6)$. It is known that $F$ is
smooth of dimension 4, see \cite{cg,ak}. Let $P=P(X)$ be the total
space of lines on $X$. We have the following diagram
$$
\xymatrix{ P\ar[r]^q\ar[d]^{p} &X\\
F & }
$$
A. Beauville and R. Donagi showed in \cite{bd} that $F$ is actually
irreducible holomorphic symplectic, i.e. $F$ is simply connected and
$\HH^0(F,\Omega_F^2)$ is generated by an everywhere nondegenerate
holomorphic 2-form $\omega$. Furthermore, they proved that the
Abel-Jacobi homomorphism
$$
\Phi=p_*q^*:\HH^4(X,\Z)(1)\to\HH^2(F,\Z)
$$
is an isomorphism of Hodge structures. The primitive cohomology of
$X$, denoted by $\HH^4(X,\Z)_{\mathrm{prim}}$, is defined to be all
elements orthogonal to $h^2$ under the intersection pairing, where
$h$ is the class of a hyperplane section on $X$. Let $b(-,-)$ be the
Beauville-Bogomolov bilinear form on $\HH^2(F,\Z)$; see
\cite{beauville2}. The primitive cohomology of $F$, denoted
$\HH^2(F,\Z)_{\mathrm{prim}}$, is defined to be all elements
orthogonal to $\Phi(h^2)$ under the bilinear form $b$. Then it is
further proved in \cite{bd} that $\Phi$ induces an isomorphism
$$
\Phi:\HH^4(X,\Z)_{\mathrm{prim}}\rightarrow\HH^2(F,\Z)_{\mathrm{prim}}
$$
and the intersection form on $X$ and the Beauville-Bogomolov form on
$F$ are related by
$$
b(\Phi(x),\Phi(y))=-(x\cdot y)_X
$$
for $x,y\in\HH^4(X,\Z)_{\mathrm{prim}}$.

\begin{defn}
Let $M$ be a free abelian group of finite rank with an integral
symmetric bilinear form $b_1:M\times M\to\Z$. Such an $M$ will be
called a \textit{lattice}. We define the kernel of the bilinear form
$b_1$ to be
$$
\ker(b_1)=\{x\in M: b_1(x,y)=0,\;\forall y\in M\}.
$$
$M$ will be called non-degenerate if $\ker (b_1)=0$. Let $x\in
M\backslash\ker(b_1)$. We define the \textit{scale}, $s_{b_1}(x)$,
of $x$ to be the largest integer that divides $b_1(m,x)$ for all
$m\in M$. The $(+)$-\textit{modification} of $(M,b_1)$ by $x$ is a
new lattice whose underlying abelian group is still $M$ while the
new bilinear form, $b_2=m^+_x(b_1)$, is given by
$$
b_2(\alpha,\beta)=b_1(\alpha,\beta)
+\frac{1}{(s_{b_1}(x))^2}b_1(x,\alpha)b_1(x,\beta).
$$
Similarly, the $(-)$-\textit{modification} of $(M,b_1)$ by $x$ is
$(M,b'_2=m^-_x(b_1))$ where the new bilinear form is given by
$$
b'_2(\alpha,\beta)=b_1(\alpha,\beta)-
\frac{1}{(s_{b_1}(x))^2}b_1(x,\alpha)b_1(x,\beta).
$$
If $x\in \ker(b_1)$, by convention, we define the scale of $x$ in
$(M,b_1)$ to be zero.
\end{defn}

\begin{lem}
Let $(M,b_1)$ be a lattice and $x\in M\backslash\ker(b_1)$. Let
$(M,b_2)$ the $(+)$-modification of $(M,b_1)$ by $x$ and $(M,b'_2)$
be the $(-)$-modification. Then the following are true

(i) The modification by $x$ is the same as the modification by $ax$
for all nonzero integer $a$.

(ii) Let $s$ be the scale of $x$ in
$(M,b_1)$ and $b_1(x,x)=sc_0$ with $c_0\in\Z$. Then the scale of $x$
in $(M,b_2)$ is $s+c_0$ and
the scale of $x$ in $(M,b'_2)$ is $s-c_0$.

(iii) If $s+c_0\neq 0$, then the $(-)$-modification of $(M,b_2)$ by
$x$ is $(M,b_1)$; If $s-c_0\neq 0$ then the
$(+)$-modification of $(M,b'_2)$ by $x$ is $(M,b_1)$.

(iv) Assume that $s\pm c_0\neq 0$, then the orthogonal complement
$x^\perp$ of $x$ is independent of the bilinear forms $b_1$, $b_2$
or $b'_2$.

(v) Assume that $M$ carries a Hodge structure of weight $2k$ and
$x\in M^{k,k}$. If $b_1$ respects the Hodge structure in the sense
that
$$
b_1(M^{i,j},M^{i',j'})=0,\text{ if }i\neq j',
$$
then both $b_2$ and $b'_2$ respect the Hodge structure of $M$.
\end{lem}
\begin{proof}
For any $\alpha\in M$, we have
$$
b_2(x,\alpha)=b_1(x,\alpha)+\frac{1}{s^2}b_1(x,x)b_1(x,\alpha)
=(s+c_0)\frac{b_1(x,\alpha)}{s}
$$
When $\alpha$ runs through $M$, $\frac{b_1(x,\alpha)}{s}$ runs
through every integer. Hence the scale of $x$ in $(M,b_2)$ is
$s+c_0$. This proves the first part of (ii). The second part is
proved in a same way. The remaining statements follow quite easily
from the definitions.
\end{proof}

From now on, we will identify $\HH^4(X,\Z)$ with $\HH^2(F,\Z)$ as
free abelian groups with Hodge structures, which will be denoted
$\Lambda$. Let $\lambda_0\in\Lambda$ be the element that corresponds
to $h^2\in\HH^4(X,\Z)$, where $h$ is the class of a hyperplane
section on $X$. On $F$, the element $\lambda_0$ is nothing but the
polarization coming from the Pl\"ucker embedding of the Grassmannian
$G(2,6)$. Let $b_0:\Lambda\times\Lambda\to\Z$ be the intersection
form on $\HH^4(X,\Z)$ and $b:\Lambda\times\Lambda\to\Z$ be the
Beauville-Bogomolov form on $\HH^2(F,\Z)$.

\begin{prop}
(i) The bilinear form $b$ can be written as
$b=-m_{\lambda_0}^-(b_0)$, i.e.
\begin{equation*}
b(x,y)=b_0(x,\lambda_0)b_0(y,\lambda_0)-b_0(x,y),\quad\forall
x,y\in\Lambda.
\end{equation*}

(ii) The bilinear form $b_0$ can be written as
$b_0=-m_{\lambda_0}^{-}(b)$, i.e.
$$
b_0(x,y)=\frac{1}{4}b(x,\lambda_0)b(y,\lambda_0)-b(x,y),\quad\forall
x,y\in\Lambda.
$$
\end{prop}
\begin{proof}
This is essentially a restatement of the result of Beauville and
Donagi. The identification of $\HH^4(X,\Z)$ and $\HH^2(F,\Z)$ is via
the Abel-Jacobi isomorphism $\Phi$. We know that
$b(\lambda_0,\lambda_0)=6$ and $b(x,y)=-b_0(x,y)$ for
$x,y\in\HH^4(X,\Z)_{\mathrm{prim}}$. By Proposition 6 of \cite{bd},
we know that $\lambda_0$ is orthogonal to
$\HH^4(X,\Z)_{\mathrm{prim}}$ under $b$, then $b$ is uniquely
determined by its restriction to $\lambda_0$ and
$\HH^4(X,\Z)_{\mathrm{prim}}$. One checks that the bilinear form
given by the formula in (i) coincides with $b$ after restricting to
$\lambda_0$ and $\HH^4(X,\Z)_{\mathrm{prim}}$. This forces (i) to be
true. To prove (ii), we only need to check that the scale of
$\lambda_0$ in $(\Lambda,b)$ is equal to $2$. This can be seen
easily by considering the special case when $X$ is a Pfaffian cubic
fourfold. In this case $F=S^{[2]}$ where $S$ is a $K3$-surface of
degree 14. Under the canonical orthogonal decomposition
$\HH^2(F,\Z)\simeq\HH^2(S,\Z)\hat{\oplus}\Z\delta$, we have
$\lambda_0=2l-5\delta$ with $l$ being the class of the degree 14
polarization on $S$. Here $\delta\in\Pic(F)$ is the half of the
boundary divisor with $b(\delta,\delta)=-2$; see \cite{beauville2}.
From this, we computes that $b(\lambda_0,\Lambda)=2\Z$ and hence the
scale of $\lambda_0$ in $(\Lambda,b)$ is 2.
\end{proof}

\begin{rmk}
Let $\Lambda_0=(\lambda_0)^\perp\subset\Lambda$, which is
independent of the bilinear form. If $x\in\Lambda_0$, then
$b(x,y)=-b_0(x,y)$, for all $y\in \Lambda$.
\end{rmk}

Let $l\subset X$ be a general line on $X$. Then we have the
following splitting type of the normal bundle of $l$ in $X$, see
\cite[Proposition 6.19]{cg},
$$
\mathscr{N}_{l/X}\cong\calO^2\oplus\calO(1).
$$
We will use $\mathscr{N}_{l/X}^+$ to denote the positive sub-bundle
$\calO(1)\subset\mathscr{N}_{l/X}$. Let $S_l\subset F$ be the space
of all lines meeting $l$. It is known that $S_l$ is a smooth
surface, see \cite{voisin}. On $S_l$, there is a natural involution
$\sigma:S_l\to S_l$. If $[l']\in S_l$ is a point different from
$[l]$, then $\sigma([l'])$ is the residue line of $l\cup l'$. The
line $l$ determines a unique linear $P_l=\PP^2\subset\PP^5$ which is
the span of $l$ and the positive subbundle $\mathscr{N}_{l/X}^+$.
The intersection of $P_l$ and $X$ is given by
$$
P_l\cdot X=2l+l_0.
$$
Then $\sigma([l])=[l_0]$. The involution $\sigma$ has 16 isolated
fixed points. This number is computed in \cite[Cor. 1.7]{av}. The
quotient of $Y_l=S_l/\sigma$ is a quintic surface in $\PP^3$ with 16
ordinary double points. To apply the results from the previous
sections, we need the following lemma whose proof will be given
later.

\begin{lem}\label{key lemma}
The surface $S_l$ satisfies $\HH_1(S_l,\Z)=0$.
\end{lem}

Let $p_l:\mathscr{C}_l\to S_l$ be the total space of lines meeting
$l$ and $q_l:\mathscr{C}_l\to X$ be the natural morphism. Let
$$
\Phi_l=(p_l)_*(q_l)^*:\HH^4(X,\Z)\to\HH^2(S_l,\Z)
$$
be the associated Abel-Jacobi homomorphism and
$$
\Psi_l=(q_l)_*(p_l)^*:\HH^2(S_l,\Z)\to\HH^4(X,\Z)
$$
be the associated cylinder homomorphism. Similarly we can define the
Abel-Jacobi and cylinder homomorphisms for the Chow groups. By abuse
of notations, we will still use $\Phi_l$ and $\Psi_l$ to denote
them, i.e.
$$
\Phi_l=(p_l)_*(q_l)^*:\CH_i(X)\to \CH_{i-1}(S_l),
$$
and
$$
\Psi_l=(q_l)_*(p_l)^*:\CH_i(S_l)\to\CH_{i+1}(X).
$$
Note that the action $\sigma$ induces an involution on
$\HH^2(S_l,\Z)$ and also on $\CH_i(S_l)$ via pull-back. We will still
use $\sigma$ to denote this action. On $S_l$, we have two natural
divisor classes $g$ and $g'$. The class $g$ is the restriction of
$\lambda_0$. The class $g'$ is defined as follows. For any point
$x\in l$, all lines passing through $x$ form a curve $C_x\subset
S_l$. We define $g'$ to be the class of the curve $C_x$.

\begin{defn}
The primitive cohomology of $S_l$ is defined to be
$$
\HH^2(S_l,\Z)_{\mathrm{prim}}=\{\alpha\in\HH^2(S_l,\Z):\alpha\cup
g=0,\alpha\cup g'=0\}
$$
\end{defn}
Note that on the group $\mathrm{Br}_2(X)$, we have
$$
\Phi_l=(p_l)_*(q_l)^*:\mathrm{Br}_2(X)\rightarrow
\HH_\mathcal{D}^3(S_l,\Z(1))_\mathrm{tor} =\mathrm{Br}(S_l).
$$
Now we state the main theorem of this section.

\begin{thm}\label{prym construction}
Let notations be as above, then the following are true.

(i) The action $\sigma$ preserves primitive classes on $S_l$, i.e.
$$
\sigma:\HH^2(S_l,\Z)_{\mathrm{prim}}\to\HH^2(S_l,\Z)_{\mathrm{prim}}
$$

(ii) The Abel-Jacobi homomorphism induces an isomorphism
$$
\Phi_l: \HH^4(X,\Z)_{\mathrm{prim}}\rightarrow
\mathrm{Pr}(\HH^2(S_l,\Z)_{\mathrm{prim}},\sigma)(-1)
$$
of Hodge structures that respects the bilinear forms. This means
that for any $\alpha,\beta\in\HH^4(X,\Z)_{\mathrm{prim}}$, we have
$$
\langle\Phi_l(\alpha),\Phi_l(\beta)\rangle=-(\alpha\cdot\beta)_X.
$$

(iii) For the Chow groups, let $\mathrm{A}_1(X)$ be the Chow group
of 1-cycles on $X$ of degree 0. Let $\mathrm{A}_0(S_l)$ be the Chow
group of 0-cycles on $S_l$ of degree 0. Then the Abel-Jacobi
homomorphism also induces an isomorphism
$$
\Phi_l:\mathrm{A}_1(X)\to \mathrm{Pr}(\mathrm{A}_0(S_l),\sigma)
$$
of abelian groups. In particular, $\mathrm{A}_1(X)$ is uniquely
divisible.

(iv) For the ``Brauer groups", we have that $\Phi_l$ fits into an
exact sequence
$$
\xymatrix{
 0\ar[r] &K\ar[r] &\mathrm{Br}_2(X)\ar[r]^{\Phi_l\quad} &\mathrm{Pr}(\mathrm{Br}(S_l),\sigma)\ar[r] &0,
}
$$
where $K\cong\HH^1(G,T^2(S_l))$.

(v) The group $K\subset\mathrm{Br}_2(X)$ is independent of the
choice of the general line $l\subset X$.
\end{thm}

\begin{proof}
The main idea of the proof follows that of \cite[Theorem 3.5]{pt}.
We give a sketch here, for more details we refer to \cite{pt}. Let
$\{l_t\subset X:t\in T\}$ be a 1-dimensional family of lines on $X$
with $l_0=l$ for some closed point $0\in T$. Let $I\subset F\times
F$ be the incidence correspondence. Let
$$
I_t=I|_{S_{l_t}\times S_l}\subset \CH^2(S_{l_t}\times S_l)
$$
be the restriction of the incidence correspondence. Let
$$
I_0=\lim_{t\to 0}I_t\in\CH^2(S_l\times S_l).
$$
By definition, $I_0$ induces the homomorphism $\Phi_l\circ\Psi_l$ on
the cohomology, the Chow groups and the Brauer group. To give a
description of $I_0$, we denote $v=\frac{d}{dt}|_{t=0}\in
\HH^0(l,\mathscr{N}_{l/X})$. Assume that the incidence lines of the
pair $(l_t,l)$, i.e. lines meeting both $l$ and $l_t$ (see
\cite{relations} for a more precise definition, where we use the
terminology of secant line instead), specialize to
$E_1,\ldots,E_5\in S_l$. For each $[l']\in S_l-[l]$, we can
canonically associate a linear $\PP^3\subset\PP^5$ passing through
$x=l'\cap l$ in the following way. If $l'\neq E_i$, $\forall
i=1,\ldots,5$, then we take the linear span of $l'$, $l$ and $v_x$
to be the $\PP^3$ associated to $l'$; If $l'=E_i$ for some $i$, then
we take the corresponding $\PP^3$ to be the linear span of $l$,
$v_x$ and the $\calO(1)$-direction $\mathscr{N}^+_{E_i/X}$ of
$\mathscr{N}_{E_i/X}$. This linear $\PP^3$ will be denoted by
$\Pi_{l'}$. There are 6 lines, denoted $\{l,l',L_1,L_2,L_3,L_4\}$,
on $X$ that pass through $x$ and lie on $\Pi_{l'}$. Then $I_0$ is
generically defined by
$$
\rho:[l']\mapsto\sigma([l'])+[L_1]+[L_2]+[L_3]+[L_4].
$$
This means that $I_0$ is represented by the closure
$\bar\Gamma_\rho$ of the graph of the above multiple valued map; see
\cite[Definition 3.2]{pt}. We can also define a correspondence
$\Gamma_v$ by sending $l'$ to the lines contained in $\Pi_{l'}$,
i.e.
$$
\Gamma_v: [l']\mapsto [l']+[L_1]+[L_2]+[L_3]+[L_4].
$$
Hence we get the following key identity
\begin{equation}\label{eq key identity I0}
I_0 = \Gamma_\sigma - \Delta_{S_l} + \Gamma_v
\end{equation}
Let $\tilde{S}_l$ be the blow up of $S_l$ at the point $[l]$. Then
we get a fibration structure $\pi:\tilde{S}_l\rightarrow l$, which
extends the map $[l']\mapsto x=l'\cap l$. The exceptional divisor of
the blow-up is mapped isomorphically onto $l$. Hence we can identify
the exceptional divisor with $l$. For $x\in l$, we can define a
linear $\PP^3$ to be the linear span of $l$, $v_x$ and the
$\calO(1)$-direction $\mathscr{N}^+_{l/X}$ of $\mathscr{N}_{l/X}$.
There are again 6 lines, denoted $2l,L_{x,1},\ldots,L_{x,4}$, on $X$
passing through $x$ that lie in the linear $\PP^3$. The rule
$$
x\mapsto \sigma([l])+[L_{x,1}]+[L_{x,2}]+[L_{x,3}]+[L_{x,4}]
$$
extends $\rho$ to a 1-to-5 multiple valued map $\tilde\rho$ from
$\tilde{S}_l$ to $S_l$. When $x$ runs through $l$, the points
$\{L_{x,1},\ldots,L_{x,4}\}$ traces out a divisor $D\subset S_l$.
Hence we have
\begin{equation}\label{explicit expression for I_0}
\bar\Gamma_\rho=\{([l_1],[l_2]):[l_1]\neq [l],\,[l_2]\in\rho([l_1])
\}\cup\{[l]\}\times D \subset S_l\times S_l.
\end{equation}

\textit{Claim 1}: The divisor class of $D$ is equal to
$\sigma(C_x)$. The class $g$ can be written as $g=2C_x+\sigma(C_x)$.
[Also see Lemme 2 in \S3 of \cite{voisin}.]

This can be seen as follows. By definition we have
$D|_{C_x}+2[l]=K_{C_x}$ since we know that $C_x$ is a
$(2,3)$-complete intersection in $\PP(T_{X,x})=\PP^3$. We can embed
$C_x\subset\PP^1\times\PP^1$. We see that $\sigma(C_x)\cap C_x$
consists of points on $C_x$ which are on the same horizontal or
vertical rulings as $[l]$. This implies that $\sigma(C_x)|_{
C_x}+2[l]$ is the restriction of the $(1,1)$ class on
$\PP^1\times\PP^1$. Hence by the adjunction formula we have
$\sigma(C_x)|_{C_x}+2[l]=K_{C_x}$. This gives
$(D-\sigma(C_x))|_{C_x}=0$ for all general $x\in C_x$. This forces
$D$ to be the same as $\sigma(C_x)$. The class
$\Psi_l(C_x)\in\CH_2(X)$ is represented by the surface swept by all
lines passing through $x$. It is known that is class is $2h^2$,
where $h$ is the hyperplane, see \cite[Lemma 3.26]{pt}. Then it
follows that
\begin{align*}
 2g &= \Phi_l\circ\Psi_l(C_x)\\
  &= (I_0)_*C_x=(\bar\Gamma_\rho)_*C_x\\
  &= \sigma(C_x) + 4 C_x +D.
\end{align*}
Hence we get $g=2C_x+\sigma(C_x)$.

\textit{Claim 2:} The primitive lattice
$\HH^2(S_l,\Z)_{\mathrm{prim}}$ is simply the orthogonal complement
of $M=\Z C_x\oplus\Z C_x^\sigma$, where $C_x^\sigma=\sigma(C_x)$. As
a $G=\Z[\sigma]$-module, we have
$$
 \HH^1(G,\HH^2(S_l,\Z)_\mathrm{prim})=0.
$$

This follows from Claim 1 and Proposition \ref{module structure of
prym}. Here we note that $M$ is saturated in $\HH^2(S,\Z)$ and
$M\cong\Z[G]$. This follows from the fact that $\det(M)=-15$ which
is square-free. Indeed, we have $(C_x)^2=(C_x^\sigma)^2=1$ and
$C_x\cdot C_x^\sigma=4$ and hence $\det(M)=1\cdot 1- 4\cdot 4=-15$.

The statement (i) is automatic since
$\HH^2(S_l,\Z)_\mathrm{prim}=M^\perp$. As in the proof of
\cite[Theorem 3.5]{pt}, we can show that the action of $\Gamma_v$ on
the primitive cohomology or the Chow group of degree zero cycles is
trivial; see \cite[Lemma 3.24]{pt}. Then the above equation
\eqref{eq key identity I0} shows that $\Phi_l\circ\Psi_l=\sigma-1$
on primitive cohomology and Chow groups. The proof of the theorem
relies on a second identity $\Psi_l\circ\Phi_l=-2$ which holds on
the primitive part of both cohomology and Chow groups. This was
proved in \cite[Theorem 2.2]{pt}. Then the theorem follows from
Theorem \ref{thm on prym construction}, except (iii). Since
$\mathrm{A}_1(X)$ is generated by lines (\cite[\S4]{paranjape}, also
see \cite[Corollary 4.3]{relations}), we know that the natural
cylinder homomorphism $\mathrm{A}_0(F)\rightarrow \mathrm{A}_1(X)$
is surjective. By the theorem of Roitman \cite{roitman}, the group
$\mathrm{A}_0(F)$ is uniquely divisible. Hence $\mathrm{A}_1(X)$ is
divisible. It follows that in the isomorphism of (ii) of Theorem
\ref{thm on prym construction}, there is no 2-torsion part.

It still remains to show that the group $K$ is independent of the
choice of the line $l\subset X$. The proof is quite similar to that
of \cite[Proposition 4.7]{pt}. Let $l_1, l_2\subset X$ be two
different general lines. Set $S_i=S_{l_i}$ with the involution
$\sigma_i$, $i=1,2$. Let $\Phi_i$ and $\Psi_i$, $i=1,2$ be the
corresponding Abel-Jacobi and cylinder homomorphisms. Then the
incidence correspondence $\Gamma_{12}\subset S_1\times S_2$ induces
$$
[\Gamma_{12}]_*=\Phi_2\circ\Psi_1:\HH^2(S_1,\Z)\rightarrow\HH^2(S_2,\Z).
$$
On the primitive cohomology, the homomorphism $[\Gamma_{12}]_*$
naturally factors into
$$
\xymatrix{
 \HH^2(S_1,\Z)_\mathrm{prim}\ar[r]^{\qquad\sigma_1-1} &P_1\ar[r]^{\gamma_{12}} &P_2\ar[r] &\HH^2(S_2,\Z)_{\mathrm{prim}},
}
$$
where $P_i=\mathrm{Pr}(\HH^2(S_i,\Z)_\mathrm{prim},\sigma_i)$ and
the induced map $\gamma_{12}$ is an isomorphism. Let
$M_i\subset\Pic(S_i)$ be the natural submodule. We can compute
explicitly that
$$
Q_{M_i}\cong(\Z/5\Z)_{+}\oplus (\Z/3\Z)_{-}.
$$
By Proposition \ref{exact sequence on Prym of Brauer}, we have
$$
\xymatrix{
 0\ar[r] &T(Q_{M_1})^{\sigma_1=-1}\ar[r]\ar[d] &T(P_1)\otimes\Q/\Z\ar[r]\ar[d]^{\gamma_{12}}
 &\mathrm{Pr}(\mathrm{Br}(S_1),\sigma_1)\ar[r]\ar[d]^{\bar{\gamma}_{12}} &0\\
 0\ar[r] &T(Q_{M_2})^{\sigma_2=-1}\ar[r] &T(P_2)\otimes\Q/\Z\ar[r] &\mathrm{Pr}(\mathrm{Br}(S_2),\sigma_2)\ar[r] &0
},
$$
where the vertical arrow in the middle is an isomorphism. Hence
$\bar{\gamma}_{12}$ is surjective. The composition of
$\Phi_1:\mathrm{Br}_2(X)\rightarrow
\mathrm{Pr}(\mathrm{Br}(S_1),\sigma_1)$ and $\bar{\gamma}_{12}$ is
equal to $\Phi_2$. It follows that $\ker(\bar{\gamma}_{12})$ is of
2-torsion. Since the groups in the first column are of 3-torsion, we
know that $\ker(\bar{\gamma}_{12})$ can only have 3-torsion. This
forces $\bar{\gamma}_{12}$ to be injective and hence an isomorphism.
Then we have
\begin{align*}
 K_1 &=\ker\{\Phi_1:\mathrm{Br}_2(X)\rightarrow \mathrm{Pr}(\mathrm{Br}(S_1),\sigma_1)\}\\
 &=\ker\{\bar{\gamma}_{12}\circ\Phi_1:\mathrm{Br}_2(X)\rightarrow \mathrm{Pr}(\mathrm{Br}(S_2),\sigma_2)\}\\
 &=\ker\{\Phi_2:\mathrm{Br}_2(X)\rightarrow \mathrm{Pr}(\mathrm{Br}(S_2),\sigma_2)\}\\
 &=K_2.
\end{align*}
This finishes the proof of the Theorem.
\end{proof}

\begin{cor}
For the transcendental lattices, the Abel-Jacobi map
$$
\Phi_l:\HH^4(X,\Z)_\mathrm{tr}\rightarrow\HH^2(S_l,\Z)_\mathrm{tr}^{\sigma=-1}
$$
is an isomorphism and
$$
\mathrm{coker}\{\Psi_l:\HH^2(S_l,\Z)_\mathrm{tr}\rightarrow
\HH^4(X,\Z)_\mathrm{tr}\}\cong K
$$
as abstract groups.
\end{cor}
\begin{proof}
We have the natural identifications
$$
\mathrm{Br}_2(X)=\Hom(\HH^4(X,\Z)_\mathrm{tr},\Q/\Z),\quad
\mathrm{Pr}(\mathrm{Br}(S_l),\sigma)=\Hom(\HH^2(S_l,\Z)_\mathrm{tr},\Q/\Z)^{\sigma=-1}.
$$
The Abel-Jacobi homomorphism $\Phi_l:\mathrm{Br}_2(X)\rightarrow
\mathrm{Pr}(\mathrm{Br}(S_l),\sigma)$ is induced by
$$
\Psi_l:\HH^2(S_l,\Z)_\mathrm{tr}\rightarrow\HH^4(X,\Z)_\mathrm{tr}.
$$
By applying $\Hom(-,\Q/\Z)$, we see that $K$ is isomorphic to the
cokernel of $\Psi_l$ on the transcendental lattice. Note that
$T^2(S_l)$ and $\HH^2(S_l,\Z)_\mathrm{tr}$ are isomorphic as
$G$-modules. By definition, We have
\begin{align*}
K &=\HH^1(G,T^2(S_l))\\
&\cong\HH^1(G,\HH^2(S_l,\Z)_\mathrm{tr}) \\
&= \HH^2(S_l,\Z)_\mathrm{tr}^{\sigma=-1}/(\sigma-1)\HH^2(S_l,\Z)_\mathrm{tr}\\
&= \HH^2(S_l,\Z)_\mathrm{tr}^{\sigma=-1}/\Phi_l\Psi_l(\HH^2(S_l,\Z)_\mathrm{tr})
\end{align*}
as abstract groups. This shows that in the sequence of inclusions
$$
\Psi_l(\HH^2(S_l,\Z)_\mathrm{tr})\hookrightarrow
\HH^4(X,\Z)_\mathrm{tr}\overset{\Phi_l}\hookrightarrow
\HH^2(S_l,\Z)_\mathrm{tr}^{\sigma=-1},
$$
the quotient of the last term by the first has the same size as that
of the second by the first. Thus the second inclusion, which is the
same as $\Phi_l$, is actually an equality.
\end{proof}

It is not clear yet whether $K\neq 0$ actually happens. The
following proposition shows that one expects $K=0$ to happen
generically.
\begin{prop}
Let $X$ be a smooth cubic fourfold and $l\subset X$ a general line.
Let $Y_l$ be the quotient of $S_l$ be the involution $\sigma$. If
$\rk\Pic(Y_l)=1$ and $\rk\Pic(S_l)\leq 3$, then $K=0$, i.e.
$$
\Phi_l:\mathrm{Br}_2(X)\rightarrow\mathrm{Pr}(\mathrm{Br}(S_l),\sigma)
$$
is an isomorphism.
\end{prop}
\begin{proof}
Let $X_l$ be the blow-up of $X$ along the line $l$. Then the
projection from $l$ defines a morphism $X_l\rightarrow\PP^3$ which
realizes $X_l$ as a conic bundle over $\PP^3$. The surface $S_l$
parameterizes lines in the singular fibers and $Y_l\subset\PP^3$ is
the discriminant divisor. By assumption, $\Pic(Y_l)$ is generated by
the hyperplane class $\lambda$. It is known that
$\pi^*\lambda=C_x+C_x^\sigma$, where $\pi:S_l\rightarrow Y_l$ is the
natural double cover; see Lemme 2 in \cite[\S3]{voisin}. If the
Picard rank of $S_l$ is 2, then $\mathrm{Alg}^2(S_l)=M$ and
$M^\perp=\HH^2(S_l,\Z)_\mathrm{tr}$. In this case we already know
that $K\cong\HH^1(G,M^\perp)=0$ by Proposition \ref{module structure
of prym}.

Assume that the Picard rank of $S_l$ is 3. Then $\Pic(S_l)=\Z
C_x\oplus\Z C_x^\sigma \oplus \Z\fa$ for some class $\fa$. Since
$\sigma(\fa)+\fa=\pi_*\pi^*\fa$ is a multiple of $\pi^*\lambda$, we
have
$$
\fa+\sigma(\fa)=n(C_x+C_x^\sigma),\quad n\in\Z.
$$
We replace $\fa$ by $\fa-nC_x$ and assume that $\sigma(\fa)=-\fa$.
Hence $\Pic(S_l)\cong\Z[G]\oplus \Z_{-}$ as a $G$-module. In
particular, $\HH^i(G,\Pic(S_l))=\Z/2\Z$ for odd $i$. The
intersection form on the Picard group, with respect to the basis
chosen above, is given by
$$
A=\begin{pmatrix}
1 &4 &m\\
4 &1 &-m\\
m &-m &d
\end{pmatrix}
$$
where $m=\fa\cdot C_x$ and $d=(\fa)^2$. One computes that
$\det(A)=-5(3d+2m^2)$. Let $\{f,f^\sigma,\alpha\}$ be the dual basis
of $\Pic(S_l)^*$ and $\{\bar{f},\bar{f}^\sigma,\bar{\alpha}\}$ its
image in $\Pic(S_l)^*/\Pic(S_l)$. Let $\iota:\Pic(S_l)\rightarrow
\Pic(S_l)^*$ be the natural inclusion induced by intersection
pairing. Then we have
\begin{align*}
\iota(C_x) &= f+4f^\sigma+m\alpha\\
\iota(C_x^\sigma) &= 4f+f^\sigma-m\alpha\\
\iota(\fa) &= mf-mf^\sigma +d\alpha
\end{align*}
It follows that
$$
5(f+f^\sigma)=\iota(c_x+c_x^\sigma),\qquad (2m^2+3d)\alpha=\iota(m(C_x-C_x^\sigma)+3\fa).
$$
Then $\bar{f}+\bar{f}^\sigma$ has
order 5 and $(3d+2m^2)\bar{\alpha}=0$. If $3\nmid m$, then
$\bar{\alpha}$ is of order $N=|3d+2m^2|$ and we have
$$
\Pic(S_l)^*/\Pic(S_l)\cong (\Z/5\Z)_{+}\oplus (\Z/N\Z)_{-}.
$$
If $m=3m_0$, then $\bar{\alpha}$ is of order $N'=|d+6m_0^2|$ and
$$
\Pic(S_l)^*/\Pic(S_l)\cong (\Z/5\Z)_{+}\oplus
(\Z/N'\Z)_{-}\oplus(\Z/3\Z)_{-}.
$$
In either case, we know that the 2-primary part $T_2$ of
$\{\Pic(S_l)^*/\Pic(S_l)\}^{\sigma=1}$ is either isomorphic to
$\Z/2\Z$ (if $N$ is even) or zero (if $N$ is odd; but we will see
that this does not happen). The exact sequence
$$
\xymatrix{
 0\ar[r] &\Pic(S_l)\oplus\HH^2(S_l,\Z)_\mathrm{tr}\ar[r] &\HH^2(S_l,\Z)\ar[r] &\Pic(S_l)^*/\Pic(S_l)\ar[r] &0.
}
$$
gives an exact sequence
$$
\xymatrix{
 \cdots\ar[r] &T_2\ar[r] &\HH^1(G,\Pic(S_l))\oplus\HH^1(G,\HH^2(S_l,\Z)_\mathrm{tr})\ar[r] &\HH^1(G,\HH^2(S_l,\Z))=0.
}
$$
Since $\HH^1(G,\Pic(S_l))=\Z/2\Z$, we get $T_2=\Z/2\Z$ and
$\HH^1(G,\HH^2(S_l,\Z)_\mathrm{tr})=0$. This implies that $K=0$.
\end{proof}

\section{Proof of Lemma \ref{key lemma}}
Let $V$ be a complex vector space with $\dim V=6$ and
$G\in\Sym^3V^\ast$ such that $G=0$ defines a smooth cubic fourfold
$X\subset\PP(V)$. Let $\{e_0,e_1,\ldots,e_5\}$ be a basis of $V$ and
$\{X_0,X_1,\ldots,X_5\}$ the dual basis of $V^*$. Let
$\mathrm{G}(r,V)$ be the Grassmannian parameterizing $r$-dimensional
subspaces of $V$. On $\mathrm{G}(r,V)$, there is the canonical rank
$r$ subbundle $\mathcal{V}_r$ of the trivial bundle $V\otimes
\calO_{\mathrm{G}(r,V)}$. More generally, we will use
$\mathrm{G}(r_1,r_2,V)$ to denote the flag variety parameterizing
$V_1\subset V_2\subset V$ with $r_1=\dim V_1< r_2=\dim V_2$. In the
particular case of $r_1=1$ and $r_2=2$, we have the following
diagram
$$
\xymatrix{
 \mathrm{G}(1,2,V)\ar[r]^{f\quad}\ar[d]_g &\mathrm{G}(1,V)=\PP(V)\\
 \mathrm{G}(2,V) &
}
$$
On $\mathrm{G}(1,2,V)$, we have the natural inclusions
$$
 f^*\mathcal{V}_1\subset g^*\mathcal{V}_2\subset V.
$$
By definition, $\Sym^3V^*$ is a quotient of $V^*\otimes V^*\otimes
V^*$. However, we can naturally identify $\Sym^3V^*$ with the
symmetric tensors in $(V^*)^{\otimes 3}$. More precisely, the
inclusion $\Sym^3V^*\subset (V^*)^{\otimes 3}$ is given by
$$
 \alpha_1\alpha_2\alpha_3\mapsto
 \frac{1}{6}(\sum\alpha_i\otimes\alpha_j\otimes\alpha_k),
$$
where $(i,j,k)$ runs through all permutations of $\{1,2,3\}$. Hence
$G\in\Sym^3V^*$ maps to an element $\tilde{G}$ in $(V^*)^{\otimes
3}$ and we write
$$
\tilde{G}=\sum a_{ijk}X_i\otimes X_j\otimes X_k.
$$
Let $G'$ be the image of $\tilde{G}$ under the natural map
$$
V^*\otimes V^*\otimes V^*\mapsto \Sym^2(V^*)\otimes V^*.
$$
If we identify $\Sym^2V^*\otimes V^*$ with $\Hom(V,\Sym^2V^*)$, then
$G'$ gives an element
$$
 G_1\in\Hom(V,\Sym^2V^*).
$$
If we identify $\Sym^2V^*\otimes V^*$ with $\Hom(\Sym^2 V,V^*)$,
then $G'$ gives an element
$$
G_2\in\Hom(\Sym^2V,V^*).
$$
We can also write
$$
G_1=\frac{1}{3}\sum_{i=0}^{5} X_i\otimes\frac{\partial G}{\partial
X_i},\qquad G_2=\frac{1}{3}\sum_{i=0}^{5} \frac{\partial G}{\partial
X_i}\otimes X_i.
$$
The canonical element $1\in\Hom(V,V)=V\otimes
V^*=\Hom(\calO_{\PP(V)},V\otimes\calO_{\PP(V)}(1))$ gives a
homomorphism
$$
e:\calO_{\PP(V)}\to V\otimes\calO_{\PP(V)}(1),\quad 1\mapsto\sum
e_i\otimes X_i.
$$
It is well known that this fits into the following Euler exact
sequence
$$
 \xymatrix{
  0\ar[r] &\calO_{\PP(V)}\ar[r]^{e\qquad} &V\otimes\calO_{\PP(V)}(1)\ar[r]
  &T_{\PP(V)}\ar[r] &0
 }
$$
There is a natural identification
$$
 \Hom(V,\Sym^2V^*)=\Hom(V\otimes\calO_{\PP(V)}, \calO_{\PP(V)}(2)).
$$
Hence $G_1$ induces a homomorphism
$$
 G_1: V\otimes\calO_{\PP(V)}\to \calO_{\PP(V)}(2).
$$
Similarly, $G_2$ also induces a homomorphism
$$
 G_2: \Sym^2V\otimes\calO_{\PP(V)}\to \calO_{\PP(V)}(1).
$$
It is easy to check that the composition
$$
\xymatrix{
  \calO_{\PP(V)}\ar[r]^{e\qquad} &V\otimes
  \calO_{\PP(V)}(1)\ar[r]^{\quad G_1\otimes1} &\calO_{\PP(V)}(3)
}
$$
is simply given by the element $G\in\Sym^3V^*$. Hence on the cubic
fourfold $X$, the above composition is 0 and hence $G_1\otimes 1$
factors through $T_{\PP(V)}|_X$. Let
$\rho:T_{\PP(V)}|_X\to\calO_X(3)$ be the resulting homomorphism.
Then all these homomorphisms fit into the following diagram
$$
\xymatrix{
 T_X\ar[r] &T_{\PP(V)}|_X\ar[r]^\rho &\calO_X(3)\\
 \mathscr{F}\ar[r]\ar[u] &V\otimes\calO_X(1)\ar[r]^{\quad G_1\otimes1}\ar[u]
 &\calO_X(3)\ar@{=}[u]\\
 \calO_X\ar[u]\ar@{=}[r] &\calO_X\ar[u]_e &
 }
$$
where all the horizontal and vertical sequences are exact. We
identify $f:f^{-1}X\to X$ with $\pi:\PP(T_{\PP(V)}|_X)\to X$, and
let $\xi$ be the relative $\calO(1)$-class. Then by the identity
\begin{align*}
\Hom(T_{\PP(V)}|_X,\calO_X(3))&=\HH^0(X,\calO_X(3)\otimes\pi_*\calO(\xi))\\
   &=\HH^0(\PP(T_{\PP(V)}|_X),\calO(\xi)\otimes\pi^*\calO_X(3)),
\end{align*}
the homomorphism $\rho$ defines an element
$\rho_1\in\HH^0(\PP(T_{\PP(V)}|_X),\calO(\xi)\otimes\pi^*\calO_X(3))$.
The vanishing of $\rho_1$ defines $\PP(T_X)\subset
\PP(T_{\PP(V)}|_X)$. This can already be seen from the above
diagram. However, we give another proof which gives more
information.

Given a point $v=[V_1\subset V_2\subset V]\in\mathrm{G}(1,2,V)$, we
take a basis $\{v_1,v_2\}$ of $V_2$ such that $V_1=\C v_1$. Let
$t\in \mathbb{A}^1$ be an affine coordinate, then
\begin{equation}\label{expansion of G}
G(v_1+tv_2)=G(v_1)+t\sum_{i=0}^{5}\frac{\partial G}{\partial
X_i}(v_1)X_i(v_2) +t^2\sum_{i=0}^{5}X_i(v_1)\frac{\partial
G}{\partial X_i}(v_2) +t^3G(v_2).
\end{equation}
If $v\in f^{-1}X$, then $G(v_1)=0$. Hence the line $l_{V_2}$,
defined by $V_2\subset V$, is tangent to $X$ at the point $\pi(v)\in
X$ if and only if
\begin{equation}\label{second order vanishing}
\sum_{i=0}^{5}\frac{\partial G}{\partial X_i}(v_1)X_i(v_2)=0.
\end{equation}
Note that the truth of the above equality is independent of the
choice of the basis $\{v_1,v_2\}$. Tracing the definition of
$\rho_1$, we see that equation \eqref{second order vanishing} is the
same as $\rho_1=0$. This can be carried out explicitly as follows.
The composition
$$
\xymatrix{
 \pi^*(\mathcal{V}_1|_X\otimes \mathcal{V}_1|_X)\otimes
 (g^*\mathcal{V}_2)|_{f^{-1}X} \ar[r] &V\otimes V\otimes V\ar[r]^{\tilde{G}}
 &\calO_{\PP(T_{\PP(V)}|_X)}
}
$$
factors through
$\pi^*(\mathcal{V}_1|_X\otimes\mathcal{V}_1|_X)\otimes
(g^*\mathcal{V}_2/f^*\mathcal{V}_1)|_{f^{-1}X}$. The resulting
homomorphism is exactly
$$
\sum_{i=0}^5\frac{\partial G}{\partial X_i}\otimes X_i:
\pi^*(\mathcal{V}_1|_X\otimes\mathcal{V}_1|_X)\otimes
(g^*\mathcal{V}_2/f^*\mathcal{V}_1)|_{f^{-1}X} \rightarrow \calO.
$$
Note that $\mathcal{V}_1=\calO_{\PP(V)}(-1)$,
$g^*\mathcal{V}_2/f^*\mathcal{V}_1=\calO_f(-1)$ where $\calO_f(1)$
is the relative $\calO(1)$ bundle of
$$
f:\mathrm{G}(1,2,V)=\PP(V/\mathcal{V}_1)\rightarrow \PP(V)
$$
Since $T_{\PP(V)}=(V/\mathcal{V}_1)\otimes \mathcal{V}_1^*$, we get
$$
\calO_f(1)=\calO(\xi)\otimes f^*\calO_{\PP(V)}(1).
$$
Hence we have
\begin{align*}
\sum_{i=0}^5\frac{\partial G}{\partial X_i}\otimes X_i &\in\Hom(
\pi^*(\mathcal{V}_1|_X\otimes\mathcal{V}_1|_X)\otimes
(g^*\mathcal{V}_2/f^*\mathcal{V}_1)|_{f^{-1}X}, \calO)\\
&=
\Hom(\pi^*\calO_X(-2)\otimes(\calO(-\xi)\otimes\pi^*\calO(-1)),\calO)\\
&=\HH^0(\PP(T_{\PP(V)}|_X),\calO(\xi)\otimes\pi^*\calO_X(3)),
\end{align*}
and this element is exactly $\rho_1$. Assume that $\rho_1=0$, then
equation \eqref{expansion of G} tells us that $l_{V_2}$ intersects
$X$ with multiplicity at least 3 at the point $[v_1]\in X$ if and
only if
$$
\sum_{i=0}^{5}X_i(v_1)\frac{\partial G}{\partial X_i}(v_2)=0.
$$
Proceed in a similar way, we see that the above condition is the
same as a vanishing of some element
$$
\rho_2\in\HH^0(\PP(T_X),\calO(2\xi)\otimes\pi^*\calO_X(3)).
$$
If we further assume $\rho_2=0$, then $l_{V_2}$ is contained in $X$
if and only if some element
$$
\rho_3\in\HH^0(\PP(T_X),\calO(3\xi)\otimes\pi^*\calO_X(3))
$$
vanishes. These facts can summarized in the following
\begin{prop}
Let $X\subset\PP(V)$ be a smooth cubic fourfold defined by $G=0$ as
above. Let $F$ be the variety of lines on $X$ with $p:P\to F$ being
the total family of lines and $q:P\to X$ the natural morphism. Then
there is a natural closed immersion $P\subset\PP(T_X)$ and $P$ is a
$(2\xi+3\tilde{h},3\xi+3\tilde{h})$-complete intersection, where
$\xi$ is the relative $\calO(1)$ class of $\pi:\PP(T_X)\to X$ and
$\tilde{h}=\pi^*h$ is the pull back of the hyperplane class.
\end{prop}

Let $l\subset X$ be a general line and $S_l$ the surface
parameterizing all lines meeting $l$. The line $l$ itself defines a
closed point $[l]\in S_l$. It is known that $S_l$ is smooth. Let
$\tilde{S}_l$ be the blow-up of $S_l$ at the point $[l]$. Then there
is a morphism $\pi_l:\tilde{S}_l\to l$, $[l']\mapsto l'\cap l$. Then
$\tilde{S}_l$ is naturally identified with $q^{-1}(l)$.
\begin{cor}
There is a natural closed immersion $\tilde{S}_l\subset\PP(T_X|_l)$ such
that $\tilde{S}_l$ is a $(2\xi+3f,3\xi+3f)$-complete intersection
and $\pi_l=\pi|_{\tilde{S}_l}$, where $f=\pi^*[pt]$.
\end{cor}

To show that $\HH_1(S_l,\Z)=0$, it suffices to show that
$\HH_1(S^\circ,\Z)=0$, where $S^\circ=S_l-[l]$. The above corollary
allows us to prove Lemma \ref{key lemma} using some generalized
version of the Lefschetz Hyperplane Theorem for quasiprojective
varieties, see \S2.2 of \cite{gm}.

\begin{proof}[Proof of Lemma \ref{key lemma}]
Let $Z=\PP(T_X|_l)$. Since $l\subset X$ is general, we have
$T_X|_l=\calO_l(2)\oplus\calO_l(1)\oplus\calO_l^2$. This canonically
gives
$$
Z_1=\PP(\calO_l(2))\subset Z_2=\PP(\calO_l(2)\oplus\calO_l(1))\subset Z.
$$
Easy computation shows
\begin{align*}
\HH^0(Z,\calO(2\xi+3f)) &=\HH^0(l,\pi_*(\calO(2\xi+3f)))\\
 &=\HH^0(l,\Sym^2(\Omega_X^1|_l)\otimes\calO_l(3))\\
 &=\HH^0(l, \calO_l(-1)\oplus\calO_l\oplus\calO_l(1)^3\oplus\calO_l(2)^2\oplus\calO_l(3)^3)\\
 &=\C^{25}.
\end{align*}
This means $|2\xi+3f|\cong\PP^{24}$.

\textit{Claim 1}: The base locus of $|2\xi+3f|$ is $Z_1$ and this
complete linear system induces an immersion $\varphi:Z-Z_1=U\to
\PP^{24}$.

First note that $Z_1\cong\PP^1$ is the section of $\pi:Z\to l$ that
corresponds to $\Omega_X^1|_l\twoheadrightarrow\calO_l(-2)$. This
means that $\xi|_{Z_1}=\calO(-2)$. Hence
$\calO(2\xi+3f)|_{Z_1}\cong\calO_{\PP^1}(-1)$. Hence $Z_1$ is
contained in the base locus of $|2\xi+3f|$. Let $z_1,z_2\in Z-Z_1$
be two arbitrary points. The natural quotient homomorphism
$$
\Omega_X^1|_l\twoheadrightarrow\calO_l^2
$$
defines a morphism $pr_{Z_2}:Z-Z_2\to\PP(\calO_l^2)=l\times\PP^1$.
Let $p_2:l\times\PP^1\to\PP^1$ be the projection to the second
factor. Assume that $\pi(z_1)\neq\pi(z_2)$ and either $z_1,z_2\notin
Z_2$, $p_2\circ pr_{Z_2}(z_1)\neq p_2\circ pr_{Z_2}(z_2)$ or one of
the $z_i$'s is in $Z_2$, then we can find a section
$C\cong\PP^1\subset Z$ that corresponds to
$\Omega_X^1|_l\twoheadrightarrow \calO(1)$, such that $C$ passes
through both $z_1$ and $z_2$. For such a section, we have
$\calO(2\xi+3f)|_{C}\cong\calO(5)$. One checks that the image of the
restriction map
$$
\HH^0(Z,\calO(2\xi+3f))\to \HH^0(C,\calO_C(5))
$$
defines a closed immersion of $C$. If $z_1,z_2\notin Z_2$,
$\pi(z_1)\neq\pi(z_2)$ and $p_2\circ pr_{Z_2}(z_1)=p_2\circ
pr_{Z_2}(z_2)$, then we can find a section $C$ that corresponds to
$\Omega_X^1|_l\twoheadrightarrow\calO$ such that $C$ passes through
both $z_1$ and $z_2$. In this case
$\calO(2\xi+3f)|_{C}\cong\calO(3)$ and the restriction map
$$
\HH^0(Z,\calO(2\xi+3f))\to \HH^0(C,\calO_C(3))
$$
is surjective. If $z_1,z_2\in Z_2$, then we can take a section $C$
such that $\xi|_C\cong\calO(-1)$. In this case, we again have that
$$
\HH^0(Z,\calO(2\xi+3f))\to \HH^0(C,\calO_C(1))
$$
is surjective. If $z_1,z_2$ are in the same fiber, we simply take
$C$ to be the line the the fiber connecting them. In this case, if
$C$ does not meet $Z_1$, then the restriction map
$$
 \HH^0(Z,\calO(2\xi+3f))\to \HH^0(C,\calO_C(2))
$$
is surjective. If $C$ meets $Z_1$ then the image of the above map
consists of all sections vanishing at $Z_1\cap C$. In summary, there
is always a rational curve $C$ passing through $z_1,z_2$ such that
the linear system $|2\xi+3f|$, restricted to $C$, defines an
immersion of $C-Z_1$. This show that the linear system separates
points of $Z-Z_1$. By a similar argument, one shows that for any
point $z\in Z-Z_1$ and a tangent vector $v\in T_{Z,z}$, there is a
curves $C$ such that $z\in C$ and $C$ is tangent to $v$ and
furthermore, the linear system defines an immersion of $C-Z_1$.
This shows that the linear system separates tangent vectors on
$Z-Z_1$.

\textit{Claim 2}: A general element $Y\in|2\xi+3f|$ is smooth.

By Bertini's theorem, $Y$ can only have singularities along
$Z_1\subset Y$. But we know that $\tilde{S}_l$ is a smooth
$(2\xi+3f, 3\xi+3f)$-complete intersection and
$Z_1\subset\tilde{S}_l$. This forces $Y$ to be smooth along $Z_1$
for the $Y$ appearing in the complete intersection. Hence a general
$Y$ is smooth along $Z_1$.

\textit{Claim 3}: For a general element $Y\in|2\xi+3f|$, we have
$\HH_1(Y,\Z)=0$.

Let $\tilde{Z}=Bl_{Z_1}(Z)$ be the blow up of $Z$ along $Z_1$. Then
we have a diagram
$$
\xymatrix{
 E_1\ar[r]\ar[d]^{\sigma_1} &\tilde{Z}\ar[d]^\sigma\\
 Z_1\ar[r] &Z
}
$$
The normal bundle of $Z_1$ in $Z$ can be described as follows
$$
\mathscr{N}_{Z_1/Z}=T_{Z/l}|_{Z_1}=\calO_{Z_1}(\xi)\otimes
(\pi^*T_X|_l/\calO(-\xi))|_{Z_1}\cong\calO(-1)\oplus\calO(-2)^2.
$$
This implies that $E_1=\PP(\calO(-1)\oplus\calO(-2)^2)$. Let $\xi_1$
be the relative $\calO(1)$ class of
$E_1=\PP(\calO(-1)\oplus\calO(-2)^2)\to Z_1=\PP^1$ and $f_1$ the
class of a fiber. Then we have
$$
(\sigma^*(2\xi+3f)-E_1)|_{E_1}=\xi_1-f_1.
$$
There is a canonical section $\tilde{Z}_1=\PP(\calO(-1))\subset
E_1=\PP(\calO(-1)\oplus\calO(-2)^2)$ of $\sigma_1$. Since
$\xi_1|_{\tilde{Z}_1}=\calO(1)$, we get
$(\xi_1-f_1)|_{\tilde{Z}_1}=0$. Using the above argument of
restricting $\sigma^*(2\xi+3f)-E_1$ to rational curves, we see that
the complete linear system $|\sigma^*(2\xi+3f)-E_1|$ defines a
morphism
$$
\tilde{\varphi}:\tilde{Z}\rightarrow \PP^{24},
$$
such that $\tilde{\varphi}$ is an immersion on
$\tilde{Z}-\tilde{Z}_1$ and contracts $\tilde{Z}_1$ to a singular
point. As a $\PP^3$-bundle over $\PP^1$, the variety $Z$ has trivial
$\HH_1$, i.e. $\HH_1(Z,\Z)=0$. Since the center of the blow up
$\sigma$ is a rational curve, we have $\HH_1(\tilde{Z},\Z)=0$. Hence
we also have
$$
\HH_1(U_1,\Z)=\HH_1(\tilde{Z},\Z)=0,
$$
where $U_1=\tilde{Z}-\tilde{Z}_1$. Let $\tilde{Y}$ be the blow up of
$Y$ along $Z_1$, then $\tilde{Y}\subset U_1$ is a hyperplane section
and
$$
\HH_1(\tilde{Y},\Z)=\HH_1(Y,\Z).
$$
By the generalized Lefschetz hyperplane theorem (\S2.2 of
\cite{gm}), we get
$$
\HH_1(\tilde{Y},\Z)=\HH_1(U_1,\Z)=0.
$$
This proves Claim 3.

The combination of Claim 2 and Claim 3 implies
$\HH_1(Y^\circ,\Z)=0$, where $Y^\circ=Y-Z_1$. We note that
$\calO_{Z_1}(3\xi+3f)=\calO(-3)$ and
$\calO_{Z_2}(3\xi+3f)=\calO_{Z_2}(3Z_1)$. The base locus of
$|3\xi+3f|$ is $Z_1$ and it defines a morphism $\psi:U=Z-Z_1\to
\PP^{37}$ such that $Z_2-Z_1$ is contracted to a point. Furthermore,
$\psi|_{Z-Z_2}$ is an immersion. Let $Y'\in |3\xi+3f|$ be general,
then $\Sigma^\circ=Y^\circ\cap Y'$ is a smooth surface that is
homeomorphic to $S^\circ$. Furthermore, $\Sigma^\circ$ is a
hyperplane section on $Y^\circ$. Again, by the Lefschetz hyperplane
theorem for quasiprojective varieties, we have
$\HH_1(\Sigma^\circ,\Z)=0$. This implies $\HH_1(S^\circ,\Z)=0$,
which finishes the proof.
\end{proof}

\section{Conic bundles}
In this section we study rationally connected fourfolds which admit
a conic bundle structure over $\PP^3$. To be more precise, let $X$ a smooth
projective variety of dimension 4. Let $f:X\rightarrow B=\PP^3$ be a
flat dominant morphism. Assume that $X_b\cong\PP^1$ for a general
point $b\in B$. Such $X$ will be called a \textit{conic bundle} over
$B$. Let $\Delta\subset B$ be the degeneration divisor. It consists
of points $b\in B$ such that $X_b$ is either a broken conic or a
double line. Let $S$ be the surface parameterizing lines in the
degenerate fibers, then there is a natural degree 2 morphism
$\pi:S\to \Delta$. This induces an involution $\sigma:S\rightarrow
S$. For simplicity, we make the following assumptions.

\begin{assm}\label{assumptions on conic bundle}
(1) The surface $S$ is smooth and irreducible with $\HH_1(S,\Z)=0$;
(2) The involution $\sigma$ has at most finitely many isolated fixed
points; (3) The degeneration divisor $\Delta$ is of odd degree.
\end{assm}

\begin{rmk}
Actually we know that $\sigma$ has at least one fixed point.
Otherwise, $\pi:S\rightarrow\Delta$ is an \'etale double cover. But
the Lefschetz Hyperplane Theorem for the fundamental group implies
that $\Delta$ is simply connected and does not have \'etale double
covers.
\end{rmk}
By construction, we have the total family of lines in degenerate
fibers given by the following diagram
$$
\xymatrix{
 \mathscr{C}\ar[r]^q\ar[d]_p &X\\
 S &
}
$$
Let $\Phi=p_*q^*$ be the Abel-Jacobi homomorphism and $\Psi=q_*p^*$
be the cylinder homomorphism as in Section 3.3.

To understand the geometry of $X$ and $S$ better, we note that in
our case, we can always find a rank 3 vector bundle $\mathscr{E}$ on
$B$ such that $i:X\hookrightarrow\PP(\mathscr{E})$, see
\cite{sarkisov}. Let $f_1:\PP(\mathscr{E})\rightarrow B$ be the
natural projection with $f=f_1\circ i$. Let
$\xi\in\mathrm{Pic}(\PP(\mathscr{E}))$ be the Chern class of the
relative $\calO(1)$-bundle. Then the divisor class of $X$ on
$\PP(\mathscr{E})$ is equal to $2\xi+f_1^*c_1(\mathscr{L})$ for some
line bundle $\mathscr{L}$ on $B$. Hence $X$ determines, up to a
scalar, a global section
$s\in\HH^0(\PP(\mathscr{E}),f_1^*\mathscr{L}\otimes\calO_{\PP(\mathscr{E})}(2))$
and $X$ is the vanishing locus of $s$. The grassmannian
$G(2,\mathscr{E})$ can be viewed as the parameter space of lines on
$\PP(\mathscr{E})$ and $S$ naturally embeds into $G(2,\mathscr{E})$.
Put these together, we get the following diagram
\begin{equation}\label{diagram S}
 \xymatrix{
  \mathscr{C}\ar[r]^{j'\quad}\ar[d]_p &G(1,2,\mathscr{E})\ar[r]^{\quad h_1}\ar[d]^g &\PP(\mathscr{E})\ar[d]^{f_1}\\
  S\ar[r]^{j\quad} &G(2,\mathscr{E}) \ar[r]^{\quad f_2} &B
}
\end{equation}
On the Grassmannian $G(2,\mathscr{E})$, there is the natural rank 2
sub-bundle $\mathcal{V}_2$ of $f_2^*\mathscr{E}$. Let
$\xi_1=h_1^*\xi$, then we have
\begin{align*}
 \HH^0(\PP(\mathscr{E}),f_1^*\mathscr{L}\otimes\calO(2\xi)) &\subset
 \HH^0(G(1,2,\mathscr{E}),h^*f_1^*\mathscr{L}\otimes\calO(2\xi_1))\\
&=\HH^0(G(2,\mathscr{E}),g_*(g^*f_2^*\mathscr{L}\otimes\calO(2\xi_1)))\\
&=\HH^0(G(2,\mathscr{E}),f_2^*\mathscr{L}\otimes\Sym^2(\mathcal{V}_2^*))
\end{align*}
Hence the element $s$ determines a section
$s'\in\HH^0(G(2,\mathscr{E}),f_2^*\mathscr{L}\otimes\Sym^2\mathcal{V}_2^*)$.
Then by defintion, $S$ is simply the vanishing locus of $s'$. This
shows that
$$
 [S]=c_3(f_2^*\mathscr{L}\otimes\Sym^2\mathcal{V}_2^*)
$$
holds in $\HH^6(G(2,\mathscr{E}),\Z)$. Let $M\subset\HH^2(S,\Z)$ be
the saturation of the image of the restriction homomorphism
$\HH^2(G(2,\mathscr{E}),\Z)\rightarrow\HH^2(S,\Z)$.

\begin{lem}
The subgroup $M$ is actually a $G$-module and $\HH^1(G,M^\perp)=0$.
\end{lem}
\begin{proof}
First we note that $\HH^2(G(2,\mathscr{E}),\Z)$ is generated by
$\fa=c_1(\mathcal{V}_2)$ and $f_2^*h$, where $h$ is the class of a
hyperplane on $B$. It is easy to see that $\sigma(h|_S)=h|_S$. Let
$\mathcal{V}'_2=j^*\mathcal{V}_2$. Let $\Delta^0$ be the smooth
locus of $\Delta$ and $S^0=\pi^{-1}(\Delta^0)$. By sending a point
$t\in S^0$ to the singular point of $X_b$ where $b=\pi(t)$, we have
a section $\tau:S^0\rightarrow\mathscr{C}$. This section descends to
a section $\tau_0:\Delta^0\rightarrow\PP(\mathscr{E})$. This means that we have a commutative diagram
$$
\xymatrix{
 \mathscr{C}\ar[r] &\PP(\mathscr{E})\\
 S^0\ar[u]^\tau\ar[r]^{\pi} &\Delta^0\ar[u]_{\tau_0}
}
$$
The closure
$\tilde{S}$ of $\tau(S^0)$ in $\mathscr{C}$ is the blow-up of $S$ at
the fixed points of $\sigma$; the closure $\tilde{\Delta}$ of
$\tau_0(\Delta^0)$ in $\PP(\mathscr{E})$ is the minimal resolution
of $\Delta$. On the surface $S^0$, we have the following short exact
sequence
$$
\xymatrix{
 0\ar[r] &\mathscr{M}\ar[r]
 &\mathcal{V}'_2\oplus\sigma^*\mathcal{V}'_2\ar[r]
 &\mathscr{E}|_S\ar[r] &0
}
$$
The projection $\mathscr{M}\rightarrow \mathcal{V}'_2$ to the first
factor defines the section $\tau$. This means that
$\tau^*\calO_\mathscr{C}(-1)=\mathscr{M}$, where $\calO_\mathscr{C}(-1)$ is the tautological line bundle on $\mathscr{C}$. Note that $c_1(\calO_\mathscr{C}(-1))=-\xi|_\mathscr{C}$ and $\tau^*(\xi|_\mathscr{C})=\pi^*\tau_0^*\xi$. Hence we get
$$
c_1(\mathcal{V}'_2)+\sigma^*c_1(\mathcal{V}'_2)=c_1(\mathscr{E}|_S)-\tau^*(\xi|_{\mathscr{C}})
=c_1(\mathscr{E}|_S)-\pi^*\tau_0^*\xi.
$$
If we can show that $\tau_0^*\xi$ is a class that comes from
$B=\PP^3$, then we know that $\sigma^*c_1(\mathcal{V}'_2)$ is again
an element in $M$, which shows that $M$ is a $G$-module. Hence it
remains to study the class $\tau_0^*\xi$. Note that the element
\begin{align*}
s &\in\HH^0(\PP(\mathscr{E}),f_1^*\mathscr{L}\otimes \calO(2\xi))\\
 &= \HH^0(B,\Sym^2(\mathscr{E}^*)\otimes\mathscr{L})\\
 &\subset\HH^0(B,\mathscr{E}^*\otimes\mathscr{E}^*\otimes\mathscr{L})\\
 &=\Hom(\mathscr{E},\mathscr{E}^*\otimes\mathscr{L})
\end{align*}
determines a homomorphism
$$
s_0:\mathscr{E}\rightarrow\mathscr{E}^*\otimes\mathscr{L}.
$$
Now consider the following composition
$$
\xymatrix{
 \calO_{\PP(\mathscr{E})}(-1)\ar[r] &f_1^*\mathscr{E}\ar[rr]^{f_1^*s_0\quad}
 &&f_1^*(\mathscr{E^*}\otimes\mathscr{L}),
}
$$
which defines an element
$\theta\in\HH^0(\PP(\mathscr{E}),\calO(1)\otimes
f_1^*(\mathscr{E}^*\otimes\mathscr{L}))$. The surface
$\tau_0(\Delta^0)$ is defined by the vanishing of $\theta$. On $\PP(\mathscr{E})$, by twisting the Euler short exact sequence, we have the following short exact sequence
\begin{equation}\label{sequence 2}
\xymatrix{
 0\ar[r] &\calO(-1)\ar[r] &f_1^*\mathscr{E}\ar[r]
 &\mathscr{F}\ar[r] &0,
}
\end{equation}
where $\mathscr{F}=\mathcal{T}_{\PP(\mathscr{E})/B}\otimes\calO(-1)$ is locally free of rank 2. Hence on the
surface $\tau_0(\Delta^0)$, we have another short
exact sequence
\begin{equation}\label{sequence 3}
\xymatrix{
 0\ar[r] &\mathscr{F}\ar[r]
 &f_1^*(\mathscr{E}^*\otimes\mathscr{L})\ar[r] &\mathscr{L}_1\ar[r]
 &0,
}
\end{equation}
where the first injective homomorphism is induced by $f_1^*s_0$ and $\mathscr{L}_1$ is an invertible sheaf on $\Delta^0$. By
dualizing the sequence \eqref{sequence 3}, we get
\begin{equation}\label{sequence 4}
\xymatrix{
 0\ar[r] &\mathscr{L}^{-1}_1\otimes f_1^*\mathscr{L}\ar[r]
 &f_1^*\mathscr{E}\ar[r] &\mathscr{F}^*\otimes
 f_1^*\mathscr{L}\ar[r] &0,\quad\text{on }\tau_0(\Delta_0).
}
\end{equation}
Since $s$ is symmetric, we know that the above sequence is
essentially the same as \eqref{sequence 2}. This implies that
$$
\mathscr{F}\cong\mathscr{F}^*\otimes f_1^*\mathscr{L},\quad\text{on
}\tau_0(\Delta_0).
$$
By taking the first Chern classes of both sides, we get
$c_1(\mathscr{F})=f_1^*c_1(\mathscr{L})$. Together with sequence
\eqref{sequence 2}, we see that $\xi|_{\tau_0(\Delta^0)}$ is a class
that comes from $B$. Hence $\tau_0^*\xi$ comes from $B$. If we trace
the quantities above, we get
$\tau_0^*\xi=(c_1(\mathscr{L})-c_1(\mathscr{E}))|_{\Delta_0}$. From
this identity, we can get
$$
c_1(\mathcal{V}'_2)+\sigma
c_1(\mathcal{V}'_2)=(2c_1(\mathscr{E})-c_1(\mathscr{L}))|_S.
$$

We next show that $\HH^1(G,M)=\HH^1(G,M^\perp)=0$. To do this, we
only need to show $2\nmid\det(M)$. Indeed, the condition that
$2\nmid\det(M)$ implies that the torsion module $Q_M$ has no 2-torsion and hence
$\HH^i(G,Q_M)=0$ for all $i>0$; see Lemma \ref{lem 2-torsion}. Then the short exact sequence
$$
\xymatrix{
 0\ar[r] &M\oplus M^\perp\ar[r] &\HH^2(S,\Z)\ar[r] &Q_M\ar[r] &0
}
$$
gives $\HH^i(G,M\oplus M^\perp)\cong\HH^i(G,\HH^2(S,\Z))$ for all
$i>0$. Since $\sigma$ has fixed points, by Theorem \ref{thm on
surface with involution}, we have $\HH^1(G,\HH^2(S,\Z))=0$.

It remains to show that $2\nmid \det(M)$. To do this, we only need
to show that $2\nmid \det(M_0)$ for some submodule $M_0\subset M$ of
full rank. This is because we have $\det(M_0)=\det(M)|M/M_0|^2$. In
our case, we take $M_0$ to be the $\Z$-linear span of $\fa$ and
$\sigma(\fa)$, where $\fa=c_1(\mathcal{V}'_2)$. We take $\fa_1=\fa$
and $\fa_2=\fa+\sigma(\fa)$. We have already seen that
$\fa_2=2c_1(\mathscr{E}|_S)-c_1(\mathscr{L}|_S)$. Note that
$$
\beta_1|_S\cup\beta_2|_S=2([\Delta]\cup\beta_1\cup\beta_2)
$$
is even, for all $\beta_1,\beta_2\in\HH^2(B,\Z)$. It follows that
$(\fa_2)^2$ is even. Let
$d_0=(\fa_1)^2(\fa_2)^2-(\fa_1\cup\fa_2)^2$. If we can show that
$\fa_1\cup\fa_2$ is odd, then $2\nmid d_0$ and hence $d_0\neq 0$.
This would imply that $M_0$ is of rank 2 and $\det(M_0)=d_0$ is odd.
There is a natural identification $G(2,\mathscr{E})\cong
G(1,\mathscr{E}^*)$. Let $\eta$ be the relative $\calO(1)$ class of
$G(1,\mathscr{E}^*)$, then we have the following short exact
sequence
$$
\xymatrix{
 0\ar[r] &\mathcal{V}_2\ar[r] &f_2^*\mathscr{E}\ar[r] &\calO(\eta)\ar[r] &0.
}
$$
From this we deduce that
$$
c_1(\mathcal{V}_2)=f_2^*c_1(\mathscr{E})-\eta,\quad
c_2(\mathcal{V}_2)=f_2^*c_2(\mathscr{E})-f_2^*c_1(\mathscr{E})\eta+\eta^2.
$$
Then the class of $S$ in $\HH^6(G(2,\mathscr{E}),\Z)$ is given by
\begin{align*}
[S] &=c_3(f_2^*\mathscr{L}\otimes \Sym^2\mathcal{V}^*_2)\\
 &=c_3(\Sym^2\mathcal{V}^*_2)
+c_2(\Sym^2\mathcal{V}^*_2)f_2^*c_1(\mathscr{L})
+c_1(\Sym^2\mathcal{V}^*_2)f_2^*c_1(\mathscr{L})^2+f_2^*c_1(\mathscr{L})^3\\
 &=4c_1(\mathcal{V}_2^*)c_2(\mathcal{V}_2^*)
+(2c_1(\mathcal{V}_2^*)+4c_2(\mathcal{V}_2^*))f_2^*c_1(\mathscr{L})\\
 &\quad +3c_1(\mathcal{V}_2^*)f_2^*c_1(\mathscr{L})^2
+f_2^*c_1(\mathscr{L})^3\\
 &\equiv 3(\eta-f_2^*c_1(\mathscr{E}))f_2^*c_1(\mathscr{L})^2
+f_2^*c_1(\mathscr{L})^3,\mod 2
\end{align*}
Hence we have
\begin{align*}
\fa_1\cup\fa_2 &=[S]\cup(f_2^*c_1(\mathscr{E})-\eta)\cup
f_2^*(2c_1(\mathscr{E})-c_1(\mathscr{L}))\\
 &\equiv [S]\cup\eta\cup f_2^*c_1(\mathscr{L}),\mod 2\\
 &\equiv 3(\eta-f_2^*c_1(\mathscr{E}))f_2^*c_1(\mathscr{L})^3\cup\eta,\mod
2\\
 &\equiv \eta^2\cup f_2^*c_1(\mathscr{L})^3,\mod 2\\
 &\equiv c_1(\mathscr{L})^3,\mod 2
\end{align*}
Note that the degenerate locus $\Delta\subset B$ is defined by the vanishing of
$$
\wedge^3 s_0:\det(\mathscr{E})\longrightarrow \det(\mathscr{E}^*\otimes\mathscr{L}) = \det{\mathscr{E}}^*\otimes\mathscr{L}^3
$$
Hence we see that the class of $\Delta$ is given by
$[\Delta]=-2c_1(\mathscr{E})+3c_1(\mathscr{L})$. The fact that
$\deg(\Delta)$ is odd implies that $c_1(\mathscr{L})$ is odd. Hence
we conclude that $\fa_1\cup\fa_2$ is odd.
\end{proof}

\begin{cor}
The module $M$ is isomorphic to $\Z[G]$ as a $G$-module.
\end{cor}
\begin{proof}
The above proof shows that $\HH^1(G,M)=0$, which implies that as
$G$-modules we have either $M\cong\Z[G]$ or $M\cong (\Z_{+})^2$. But
the fact $d_0\neq 0$ implies that $\sigma(\fa)\neq\fa$. Hence $M$ is
not isomorphic to $(\Z_{+})^2$.
\end{proof}

As in Section 4, we use $N\subset\HH^4(X,\Z)$ to denote the
saturation of $\Sym^2\HH^2(X,\Z)$. The following lemma assures that
$\Phi(N^\perp)\subset M^\perp$ and $\Psi(M^\perp)\subset N^\perp$.
Note that $N^\perp=\HH^4(X,\Z)_\mathrm{prim}$ by definition.
\begin{lem}
The modules $M$ and $N$ are related by $\Psi(M)\subset N$ and
$\Phi(N)\subset M$.
\end{lem}
\begin{proof}
The inclusion $\Phi(N)\subset M$ is quite straight forward. First we note that
$$
N\otimes\Q=\mathrm{Im}\{\HH^4(\PP(\mathscr{E}),\Q)\rightarrow\HH^4(X,\Q)\}.
$$
Namely the classes in $N$ come from $\PP(\mathscr{E})$. We view $G(2,\mathscr{E})$ as the space of lines in the fibers of $\PP(\mathscr{E})\rightarrow B$. By considering the total space of such lines, we get an Abel-Jacobi homomorphism
$$
\tilde{\Phi}:\HH^4(\PP(\mathscr{E}),\Q)\longrightarrow \HH^2(G(2,\mathscr{E}),\Q)
$$
such that $\Phi(\alpha|_X)=\tilde{\Phi}(\alpha)|_S$. Hence it follows that $\Phi(N)$ comes from classes on $G(2,\mathscr{E})$, i.e. $\Phi(N)\subset M$.

For the inclusion $\Psi(M)\subset N$, the only non-trivial part is
$\Psi(c_1(\mathcal{V}_2)|_S)\in N$. On the total space
$\mathscr{C}$, we have the Euler exact sequence
$$
\xymatrix{
 0\ar[r] &\calO_{\mathscr{C}}\ar[r]
 &p^*(j^*\mathcal{V}_2)\otimes\calO_{\mathscr{C}}(\xi)\ar[r]
 &\mathcal{T}_{\mathscr{C}/S}\ar[r] &0,
}
$$
By taking Chern classes, we get
$$
p^*c_1(j^*\mathcal{V}_2)=c_1(\mathcal{T}_{\mathscr{C}/S})-2\xi|_\mathscr{C}.
$$
By the projection formula, we have
$$
q_*(\xi|_{\mathscr{C}})=\xi|_X\cdot q_*\mathscr{C} =\xi|_X\cdot f^*\Delta\in N.
$$
We also recall that $\Psi(c_1(\mathcal{V}_2|_S))=q_*p^*c_1(j^*\mathcal{V}_2)$. This implies that $\Psi(j^*c_1(\mathcal{V}_2))\in N$ is equivalent to
that $q_*c_1(T_{\mathscr{C}/S})\in N$.

The surface $S$ is naturally
a double cover of $\Delta$ and we get
$K_S=\pi^*K_\Delta=(d-4)\pi^*(h|_\Delta)$, where $d=\deg(\Delta)$
and $h$ is the class of a hyperplane of $\PP^3$. Hence we get
$c_1(K_\mathscr{C})=c_1(T_{\mathscr{C}/S})+p^*\pi^*(d-4)(h|_\Delta)$,
from which we conclude that $q_*c_1(T_{\mathscr{C}/S})\in N$ is
equivalent to $q_*K_{\mathscr{C}}\in N$. By the relation
$$
c_1(\mathscr{N}_{\mathscr{C}/X})- K_\mathscr{C}=-q^*K_{X},
$$
we see that it suffices to show
$q_*c_1(\mathscr{N}_{\mathscr{C}/X})\in N$. The following fact is
standard. If we have two rational curves $C_1$ and $C_2$ on a
surface $Y$ such that they meet transversally in a point $y$, then
$\mathscr{N}_{C/Y}|_{C_1}\cong\mathscr{N}_{C_1/Y}(y)$, where
$C=C_1\cup C_2$. If we globalize this and apply to the family
$\mathscr{C}/S$, we get $\mathscr{N}_{\mathscr{C}/X}\cong
p^*\pi^*\calO_\Delta(d\,h)\otimes\calO_{\mathscr{C}}(-\tau(S))$.
Thus we only need to show $q_*[\tau(S)]\in N$. Note that
$q_*[\tau(S)]=2[\tilde{\Delta}]$, where $\tilde{\Delta}$ is the
locus of singular points in the fibers of $f:X\rightarrow B$. So
every thing is reduced to showing $[\tilde{\Delta}]\in N$.

In the
proof of the previous lemma (right above the sequence \eqref{sequence 2}), we showed that $\tilde{\Delta}$ is
defined by the vanishing of
$\theta\in\HH^0(\PP(\mathscr{E}),\calO(1)\otimes
f_1^*(\mathscr{E}^*\otimes\mathscr{L}))$. We can view $\theta$ as a
homomorphism
$$
 \theta: f_1^*\mathscr{E}\otimes\calO(-\xi)\rightarrow \mathscr{L}.
$$
let $\eta':\calO(-2\xi)=\calO(-\xi)\otimes\calO(-\xi)\rightarrow
f_1^*\mathscr{E}\otimes\calO(-\xi)$ be the natural homomorphism
induced by $\calO(-\xi)\rightarrow f_1^*\mathscr{E}$. Then $X$ is
defined by the vanishing of $\theta\circ\eta'$. Let $\mathscr{F}$ be
the quotient of $\calO(-\xi)\rightarrow f_1^*\mathscr{E}$. Then on
$X$, since $\theta\circ\eta'=0$, we have an induced homomorphism
$$
 \theta':\mathscr{F}\otimes\calO_X(-\xi)\rightarrow f^*\mathscr{L},
$$
whose vanishing defines the class $[\tilde{\Delta}]$ on $X$. Hence
$$
[\tilde{\Delta}]=c_2(\mathscr{F}^*\otimes\calO(\xi)\otimes
f_1^*\mathscr{L})|_X
$$
is a class coming from $\PP(\mathscr{E})$ and hence is in $N$.
\end{proof}

\begin{lem}
The identity $\Psi\circ\Phi=-2$ holds as a homomorphism
$\HH^4(X,\Z)_\mathrm{prim}\rightarrow \HH^4(X,\Z)_\mathrm{prim}$ and
also as a homomorphism $\mathrm{A}_1(X)\rightarrow \mathrm{A}_1(X)$.
\end{lem}
\begin{proof}
The idea of the proof is similar to that of \cite[\S2]{pt}. Let $f_0:\Gamma\rightarrow X$ be either a continuous map from a real
4-dimensional topological manifold to $X$ or a morphism from an
algebraic curve to $X$. For each point $t\in\Gamma$, the lines in
$\PP(\mathscr{E})$ passing through $x=f_0(t)$ is parameterized by
$\PP(T_{\PP(\mathscr{E})/B,x})$. This shows that here is a
map/morphism $j_0:\Sigma=\PP(\mathscr{F}_0)\rightarrow
G(2,\mathscr{E})$, where
$\mathscr{F}_0=f_0^*T_{\PP(\mathscr{E})/B}$. This gives rise to the
following diagram
\begin{equation}
\xymatrix{
 D_1\cup D_2\ar[r]^{j_2}\ar[d] &X'\ar[r]^{h_2}\ar[d]^{i_1} &X\ar[d]^i\\
 P\ar[r]^{j_1}\ar[d]^{g'} &G(1,2,\mathscr{E})\ar[r]^{h_1}\ar[d]^g &\PP(\mathscr{E})\ar[d]^{f_1}\\
 \Sigma\ar[r]^{j_0} &G(2,\mathscr{E})\ar[r]^{f_2} &B
}
\end{equation}
where all the squares are fiber products. Each line on
$\PP(\mathscr{E})$ meets $X$ in two points; $D_2$ corresponds to the
point on $f_0(\Gamma)$ and $D_1$ corresponds to the remaining point.
Let $\varphi_i:D_i\rightarrow X$, $i=1,2$, be the natural maps. To
make the exposition simpler, all the following identities will be
understood to be modulo homological equivalence when $\Gamma$ is a
topological cycle and modulo rational equivalence when $\Gamma$ is
an algebraic cycle. First we note that $\varphi_2$ contracts $D_2$ onto a variety of smaller dimension and hence $\varphi_{2,*}[D_2]=0$. It follows that
$$
\varphi_{1,*}[D_1]=(h_2\circ j_2)_*([D_1]+[D_2]) =i^*(h_1\circ j_1)_*[P]
$$
is a class coming from $\PP(\mathscr{E})$.
We also easily see that
$(\varphi_1)_*[D_1]=f^*f_*[\Gamma]$, where $f=f_1\circ i:X\rightarrow B$, and
$$
(\varphi_1)_*(\xi|_{D_1})=(\varphi_1)_*[D_1]\cdot\xi
=(\pi^*\pi_*[D_1])\cdot\xi.
$$
Let $\rho_i:D_i\rightarrow \Sigma$, $i=1,2$, be the natural maps,
then $\rho_2$ is an isomorphism while $\rho_1$ is isomorphism way
from points $t\in \Sigma$ such that the corresponding line $L_t$ is
contained in $X$. If $L_t$ is contained in $X$, then
$\rho_1^{-1}(t)=L_t$. Let $E\subset D_1$ be the exceptional loci of
the map $\rho_1$. Since the restriction of $\xi|_{D_1}+[E]$ to each
$E_t$ is trivial (see Claim 1 in the proof of \cite[Theorem 2.2]{pt}), we conclude that
\begin{equation}\label{equation 1}
\xi|_{D_1}+E=\rho_1^*\fb,
\end{equation}
for some cycle class $\fb$ from $\Sigma$. It follows that
$\fb=(\rho_1)_*(\xi|_{D_1})$. As classes on $\Sigma$, we have
\begin{align*}
 (\rho_1)_*(\xi|_{D_1})+(\rho_2)_*(\xi|_{D_2})
 &=g'_*j_1^*h_1^*([X]\cdot \xi)\\
 &=g'_*(2\xi^2+c_1(\mathscr{L})\cdot\xi)\\
 &=2g'_*(\xi|_P)^2+c_1(\mathscr{L})|_\Sigma.
\end{align*}
Note that $(\rho_2)_*(\xi|_{D_2})=\nu^*(\xi|_\Gamma)$, where
$\nu:\Sigma=\PP(\mathscr{F}_0)\rightarrow \Gamma$ is natural map. On
$P=\PP(\mathcal{V}_2|_\Sigma)$, we have the following identity
$$
(\xi|_P)^2+g'^*c_1(\mathcal{V}_2|_\Sigma)(\xi|_{P})+g'^*c_2(\mathcal{V}_2|_\Sigma)=0,
$$
which implies that $g'_*(\xi|_P)^2=-c_1(\mathcal{V}_2|_\Sigma)$. On
$\Sigma$, we have the following diagram
$$
\xymatrix{
 & &Q_\lambda\ar@{=}[r] &Q_\lambda &\\
 0\ar[r] &\calO_\Sigma\ar[r]
 &\nu^*(\mathscr{E}|_\Gamma\otimes\calO_\Gamma(\xi))\ar[r]\ar[u]
 &\nu^*\mathscr{F}_0\ar[r]\ar[u] &0\\
 0\ar[r] &\calO_\Sigma\ar[r]\ar@{=}[u]
 &\mathcal{V}_2|_\Sigma\otimes\nu^*\calO_\Gamma(\xi)\ar[r]\ar[u] &\calO(-\lambda)\ar[r]\ar[u]
 &0
}
$$
where $\lambda$ is the class of the relative $\calO(1)$-bundle of
$\Sigma=\PP(\mathscr{F}_0)\rightarrow \Gamma$. The last row gives
$$
c_1(\mathcal{V}_2|_\Sigma)+2\nu^*(\xi|_\Gamma)+\lambda=0.
$$
Combine the above identities, we get
\begin{align*}
\fb &=(\rho_1)_*(\xi|_{D_1})\\
    &=2g'_*(\xi|_P)^2+c_1(\mathscr{L})|_\Sigma-(\rho_2)_*(\xi|_{D_2})\\
    &=-2c_1(\mathcal{V}_2|_\Sigma)+c_1(\mathscr{L})|_\Sigma-\nu^*(\xi|_\Gamma)\\
    &=2\lambda+3\nu^*(\xi|_\Gamma)+c_1(\mathscr{L})|_\Sigma.
\end{align*}
Note that if we write $\gamma=(f_0)_*[\Gamma]$, then
$\Psi\circ\Phi(\gamma)=(\varphi_1)_*[E]$. Hence we have
\begin{align*}
 \Psi\circ\Phi(\gamma)
 &=\varphi_{1*}\rho_1^*\fb-\varphi_{1*}(\xi|_{D_1})\\
 &=(\varphi_{1*}\rho_1^*\fb +\varphi_{2*}\rho_2^*\fb)
 -\varphi_{2*}\rho_2^*\fb -\varphi_{1*}(\xi|_{D_1})\\
 &=i^*h_{1*}j_{1*}g'^*\fb-\varphi_{2*}\rho_2^*\fb
 -\varphi_{1*}(\xi|_{D_1})\\
 &=-\varphi_{2*}\rho_2^*\fb,\mod\text{ classes from
 }\PP(\mathscr{E})\\
 &=-\varphi_{2*}\rho_2^*(2\lambda), \mod\text{ classes from
 }\PP(\mathscr{E})\\
 &=-2\gamma.
\end{align*}
In other words, $\Psi\circ\Phi(\gamma)+2\gamma$ is always a class
from $\PP(\mathscr{E})$. By linearity, this is true for any cycle
class $\gamma$ in $\HH^4(X,\Z)$ or $\mathrm{A}_1(X)$. Note that
$\HH^2(\PP(\mathscr{E}),\Q)\rightarrow \HH^2(X,\Q)$ and
$\HH^4(\PP(\mathscr{E}),\Q)\rightarrow N\otimes\Q$ are all
isomorphisms. Let $\gamma\in\HH^4(X,\Z)_\mathrm{prim}$, then
$\Psi\circ\Phi(\gamma)+2\gamma$ is again in
$\HH^4(X,\Z)_\mathrm{prim}$, this forces it to be zero. Hence
$\Psi\circ\Phi(\gamma)=-2\gamma$. Similarly, we have
$\Psi\circ\Phi(\gamma)=-2\gamma$ for all $\gamma\in\mathrm{A}_1(X)$.
\end{proof}

\begin{lem}
The identity $\Phi\circ\Psi=\sigma-1$ holds as an endomorphism of
$M^\perp$ and also as an endomorphism of $\mathrm{A}_0(S)$.
\end{lem}
\begin{proof}
The total space $\mathscr{C}$ can be viewed as an element in
$\CH_3(S\times X)$. Let $\mathscr{C}^t\subset X\times S$ be the
transpose of $\mathscr{C}$. Then $\Phi\circ\Psi$ is induced by
$p_{13*}(p_{12}^*\mathscr{C}\cdot
p_{23}^*\mathscr{C}^t)\in\CH_2(S\times S)$, where $p_{ij}$ are the projections from $S\times X\times X$ to corresponding factors. The cycle
$p_{12}^*\mathscr{C}$ is represented by
$$
C_1=\{([l],x,[l'])\in S\times X\times S:x\in l\}
$$
and $p_{23}^*\mathscr{C}^t$ is represented by
$$
C_2=\{([l],x,[l'])\in S\times X\times S:x\in l'\}.
$$
Note that $C_1\cap C_2=\Gamma_1\cup\Gamma_2$, where
$$
\Gamma_1=\{([l],x,\sigma([l])):x\subset l\cup\sigma(l)\},\quad
\Gamma_2=\{([l],x,[l]):x\in l\}.
$$
The intersection $C_1\cdot C_2=\gamma_1+\gamma_2$ where
$\gamma_1=[\Gamma_1]$ and $\gamma_2\in\CH_2(\Gamma_2)$. Hence
$$
p_{13*}(C_1\cdot C_2)=\Gamma_\sigma+a\Delta_S,
$$
for some $a\in\Z$, where $\Gamma_\sigma$ is the graph of $\sigma$.
Pick a general point $[l]\in S$, then we have
$\Phi\circ\Psi([l])=\sigma([l])+a[l]$. We use the fact that $\Phi$
and $\Psi$ are transpose to each other and get
$$
1+a=[S]\cdot\Phi(l)=\Psi([S])\cdot l=f^*[\Delta]\cdot
l=[\Delta]\cdot f_*l=[\Delta]\cdot 0=0.
$$
Thus we get $a=-1$. Hence we have $\Phi\circ\Psi=\sigma-1$.
\end{proof}
By the above lemmas and Theorem \ref{thm on prym construction}, we
easily get the following theorem.
\begin{thm}
Let $f:X\subset\PP(\mathscr{E})\rightarrow B=\PP^3$ be a standard
conic bundle with the associated double cover
$\pi:S\rightarrow\Delta$ that satisfies Assumption \ref{assumptions
on conic bundle}. Let $\mathscr{C}\subset S\times X$ be the total
space of lines with the natural projections
$p:\mathscr{C}\rightarrow S$ and $q:\mathscr{C}\rightarrow X$. Let
$j:S\hookrightarrow G(2,\mathscr{E})$ be the natural inclusion and
$M\subset \HH^2(S,\Z)$ be the saturation of the image of
$j^*:\HH^2(G(2,\mathscr{E}),\Z)\rightarrow\HH^2(S,\Z)$. Let
$\Phi=p_*q^*$ be the Abel-Jacobi homomorphism and $\Psi=q_*p^*$ be
the cylinder homomorphism as above. Then the following are true.

(i) The Abel-Jacobi homomorphism induces an isomorphism
$$
\Phi:\HH^4(X,\Z)_\mathrm{prim}\rightarrow\mathrm{Pr}(M^\perp,\sigma)(-1),
$$
which is compatible with the bilinear forms, namely
$$
\langle\Phi(x),\Phi(y)\rangle=-(x\cdot y)_X,\quad\forall
x,y\in\HH^4(X,\Z)_\mathrm{prim}.
$$

(ii) There is a canonical isomorphism
$$
\mathrm{A}_1(X)\cong\mathrm{Pr}(\mathrm{A}_0(S),\sigma)\oplus
(\text{2-torsion}),
$$
such that $\Phi:\mathrm{A}_1(X)\rightarrow\mathrm{A}_0(S)$ is
identified with the projection to the first factor.

(iii) There is a short exact sequence
$$
\xymatrix{
 0\ar[r] &K\ar[r] &\mathrm{Br}_2(X)\ar[r]^\Phi
 &\mathrm{Br}(S)\ar[r] &0,
}
$$
where $K=\HH^1(G,T^2(S))$.
\end{thm}
\begin{proof}
To apply Theorem \ref{thm on prym construction} in (iii), we only
need to show that $2\nmid \det(N)$. Since $\det(N)=\pm\det(N^\perp)$
and $N^\perp=\HH^4(X,\Z)_\mathrm{prim}\cong
P=\mathrm{Pr}(M^\perp,\sigma)$, we only need to show that $\det(P)$
is not divisible by $2$. By Proposition \ref{determinant of prym},
we have $\det(P)=\pm 2\frac{q^2}{q'}$, where $q=|(Q_M)^{\sigma=-1}|$
and $q'=\det(M^{\sigma=-1})$. By assumption $q$ is not divisible by
$2$. Hence we only need to show that $2$ divides $q'$. Since
$M\cong\Z[G]$, we can find $x\in M$ such that
$M^{\sigma=-1}=\Z(\sigma(x)-x)$. Hence
$q'=(\sigma(x)-x)^2=2(x^2-x\cdot\sigma(x))$, which is divisible by
2.
\end{proof}

\begin{rmk} (a) It is quite likely that the assumption
$\HH_1(S,\Z)=0$ holds automatically. If
$\Sym^2(\mathcal{V}_2^*)\otimes f_2^*\mathscr{L}$ were an ample
vector bundle on $G(2,\mathscr{E})$, then $\HH_1(S,\Z)=0$ by a
theorem of Sommese \cite{sommese} since $S$ is the vanishing of a
section of an ample vector bundle. Unfortunately, in our case the
vector bundle $\Sym^2(\mathcal{V}_2^*)\otimes f_2^*\mathscr{L}$ is
never ample. However, we still expect that $\HH_1(S,\Z)=0$ holds
true.

(b) If $X$ is very general such that $\Pic(S)=M$, then $K=0$
and $\Phi:\mathrm{Br}_2(X)\rightarrow
\mathrm{Pr}(\mathrm{Br}(S),\sigma)$ is an isomorphism. It would be
very interesting to see if the group $K$ is always trivial or not.

(c) The case of cubic fourfolds can be viewed as a special case of
the conic bundle case as follows. Let $X\subset\PP^5$ be a smooth
cubic fourfold and $l\subset X$ a general line. Let $X_l$ be the
blow-up of $X$ along $l$. Then the projection from the line $l$
defines a conic bundle structure $\pi:X_l\rightarrow\PP^3$. The
surface parameterizes lines in broken conics is simply $S_l$ and the
degeneration divisor $\Delta\subset\PP^3$ is a quintic surface with
16 nodes. Then the above theorem applies to this situation.
\end{rmk}

\end{document}